\documentclass[11pt,reqno]{amsart}
\usepackage{amssymb,amsmath,tabularx,mathrsfs,mathbbol,yfonts}
\usepackage{amsthm,verbatim}
\usepackage[bookmarks=true]{hyperref}
\usepackage{pstricks,pst-node,pst-plot}
\usepackage[all]{xy}
\usepackage{geometry}
\usepackage[vcentermath]{youngtabMW}

\newtheorem{thm}{Theorem}[section]
\newtheorem{lem}[thm]{Lemma}
\newtheorem{prop}[thm]{Proposition}
\newtheorem{cor}[thm]{Corollary}
\theoremstyle{definition}
\newtheorem{defn}[thm]{Definition}
\newtheorem{notation}[thm]{Notation}
\newtheorem{eg}[thm]{Example}

\newtheorem{rem}[thm]{Remark}
\usepackage{graphics} %for \scalebox used for \boxtimes in displayed equation

\numberwithin{equation}{section}

\newcounter{thmlistcnt}
\newenvironment{thmlist}%
	{\setcounter{thmlistcnt}{0}%
	\begin{list}{\emph{(\roman{thmlistcnt})}}{%
		\usecounter{thmlistcnt}%
		\setlength{\topsep}{0pt}%
		\setlength{\leftmargin}{0pt}%
		\setlength{\itemsep}{0pt}%
		\setlength{\labelwidth}{17pt}
		\setlength{\itemindent}{30pt}}%
	}%
	{\end{list}}%

\renewcommand{\t}{\mathfrak{t}}

\newcommand{\TT}{\mathscr{T}}
\newcommand{\T}{\mathrm{T}}
\renewcommand{\O}{\mathcal{O}}

\newcommand{\mbf}[1]{\boldsymbol{#1}}
\newcommand{\sgn}{\mathrm{sgn}}

\newcommand{\Hom}{\operatorname{Hom}}

\newcommand{\sym}[1]{\mathfrak{S}_{#1}}
\newcommand{\alt}[1]{\mathfrak{A}_{#1}}
\newcommand{\Ind}{\operatorname{\mathrm{Ind}}}
\newcommand{\Inf}{\operatorname{\mathrm{Inf}}}
\newcommand{\Res}{\operatorname{\mathrm{Res}}}
\newcommand{\F}{F}

\newcommand{\C}{\mathscr{C}}

\newcommand{\RP}{\mathscr{R}\hskip-0.5pt\mathscr{P}}
\renewcommand{\P}{\mathscr{P}}
\newcommand{\bbullet}{\hskip0.75pt\bullet\hskip0.75pt}
\newcommand{\N}{\mathbb{N}}

\newcommand{\Proj}{\mathrm{Proj}}
\newcommand{\List}{\mathrm{List}}
\newcommand{\GL}{\mathrm{GL}}
\newcommand{\Sym}{\mathrm{Sym}}
\newcommand{\End}{\mathrm{End}}

\newcommand{\largeboxtimes}{\scalebox{1.2}{$\boxtimes$}}
\newcommand{\pdot}{}

\newcommand{\LambdaABR}{\Lambda\bigl((\alpha|\beta),\rho\bigr)}
\newcommand{\Plambda}{P^\lambda} %more consistent with D^\lambda, S^\lambda
\newcommand{\Palpha}{P^\alpha}
\newcommand{\Pbeta}{P^\beta}

\newcommand{\sourced}{source.}
\newcommand{\vertex}{vertex }
\renewcommand{\negthinspace}{\hskip-0.5pt}
\newcommand{\vthinspace}{\hskip0.5pt}

\begin{document}
\title[On Signed $p$-Kostka numbers]{On signed Young permutation modules
and signed $p$-Kostka numbers}

\author{Eugenio Giannelli}
\address[E. Giannelli]{Department of Mathematics, University of Kaiserslautern,
P.O. Box 3049, 67655 Kaiserslautern, Germany}
\email{gianelli@mathematik.uni-kl.de}

\author{Kay Jin Lim}
\address[K. J. Lim]{Division of Mathematical Sciences, Nanyang Technological University, SPMS-MAS-03-01, 21 Nanyang Link, Singapore 637371.}
\email{limkj@ntu.edu.sg}

\author{William O'Donovan}
\address[W. O'Donovan]{Department of Mathematics, Royal Holloway, University of London, United Kingdom.}
\email{william.odonovan.2014@live.rhul.ac.uk}

\author{Mark Wildon}
\address[M. Wildon]{Department of Mathematics, Royal Holloway, University of London, United Kingdom.}
\email{mark.wildon@rhul.ac.uk}

\thanks{The first author is supported by the ERC advanced grant 291512 and the London Mathematical Society Postdoctoral Mobility Grant PMG14-15 02.
The second author is supported by Singapore Ministry of Education AcRF Tier 1 grant RG13/14. Part of this work was done while the first and fourth authors visited
National University of Singapore in December 2014; this visit was supported by
London Mathematical Society grant 41406 and funding from
National University of Singapore
and Royal Holloway, University of London. We thank Kai Meng Tan for his valuable suggestions.}

\begin{abstract}
We prove the existence and main properties of signed Young modules
for the symmetric group, using only basic facts about symmetric group representations and the
Brou{\'e} correspondence. We then
prove new reduction theorems for the signed $p$-Kostka numbers,
defined to be the multiplicities of signed Young modules as direct summands
of signed Young permutation modules. We end by classifying the indecomposable signed Young permutation modules
and determining their endomorphism algebras.
\end{abstract}

\maketitle
\thispagestyle{empty}

\section{Introduction}

Let $F$ be a field of odd prime characteristic $p$
and let $\sym{n}$ denote the symmetric group of degree $n$.
In this article we investigate the modular structure of the $p$-permutation $F\sym{n}$-modules defined
 by inducing a linear representation of a Young subgroup of $\sym{n}$~to~$\sym{n}$.

Let $\P^2(n)$ be the set  of all pairs of partitions
$(\alpha|\beta)$ such that $|\alpha|+|\beta|=n$. For $(\alpha |\beta)\in \P^2(n)$,
the \textit{signed Young permutation module} $M(\alpha | \beta)$ is the $F\sym{n}$-module defined by
\begin{equation} \label{eq:signedYoung}
M(\alpha | \beta)=\Ind_{\sym{\alpha}\times\sym{\beta}}^{\sym{n}}\big(F(\sym{\alpha})\boxtimes \sgn({\sym{\beta}})\big).
\end{equation}
In \cite[page 651]{SDonkin}, Donkin defines a \emph{signed Young module} to be an indecomposable
direct summand of a signed Young permutation module and proves the following theorem.
% The following theorem, proved in \cite{SDonkin}, gives
%their main properties.

\newcommand{\DonkinText}{There exist indecomposable $F\sym{n}$-modules $Y(\lambda | p\mu)$
for $(\lambda | p \mu) \in \P^2(n)$ with the following properties:
\begin{thmlist}
\item if $(\alpha|\beta) \in \P^2(n)$ then $M(\alpha|\beta)$ is
isomorphic to a direct sum of modules $Y(\lambda|p\mu)$
for $(\lambda|p\mu) \in \P^2(n)$ such that $(\lambda|p\mu) \unrhd (\alpha|\beta)$,
\item $[M(\lambda|p\mu) : Y(\lambda|p\mu)] = 1$,
\item if $\lambda = \sum_{i=0}^r p^i\lambda(i)$ and $\mu = \sum_{i=0}^{r-1} p^i\mu(i)$ are the $p$-adic
expansions of $\lambda$ and $\mu$, as defined in~\eqref{eq:padic}, then
$Y(\lambda | p \mu)$ has as a vertex a Sylow $p$-subgroup of the Young subgroup~$\sym{\rho}$,
where~$\rho$ is the partition of $n$ having exactly $|\lambda(i)| + |\mu(i-1)|$ parts of size $p^i$
for each~$i \in \{0,\ldots,r\}$.
\end{thmlist}}

\begin{thm}[Donkin \protect{\cite{SDonkin}}]\label{T:Donkin}
\DonkinText
\end{thm}

Here $(\lambda|p\mu) \unrhd (\alpha|\beta)$ refers to the dominance order on $\P^2(n)$, as defined
in Section~\ref{Sec:PCs} below and, in (iii), $\mu(-1)$ should be interpreted as the partition of $0$.
%Note that if $\mu = \varnothing$ and $\lambda = \lambda(0)$ then (iii) asserts that $Y(\lambda | \varnothing)$
%is projective.

Donkin's definition of signed Young modules and his proof of his theorem use the Schur superalgebra.
In Section~\ref{S:YoungConstruction} we
give an independent
proof using only basic facts about symmetric group representations and the Brou{\'e} correspondence
for $p$-permutation modules; our proof
shows that the $Y(\lambda|p\mu)$ may be defined by Definition~\ref{D:signedYoung}.
(Theorem~\ref{T:Donkin} characterizes the signed Young module $Y(\lambda|p\mu)$ as the unique
summand of $M(\lambda|p\mu)$ appearing in $M(\alpha|\beta)$ only if $(\lambda|p\mu) \unrhd (\alpha|\beta)$,
so the two definitions are equivalent).
 As a special case we obtain the
existence and main properties
of the \emph{Young modules}, which we define by
$Y^\lambda = Y(\lambda | \varnothing)$. These are precisely the indecomposable summands
of the \emph{Young permutation modules} $M^\alpha = M(\alpha|\varnothing)$.
We state this result, and discuss the connection with \cite{ErdSchr}, and with
the original definition of Young modules via the Schur algebra~\cite{JamesYoung},
in Subsection~\ref{S:EqvDefs}.
%Further connections with earlier results obtained using Schur algebras and Schur super algebras
%are discussed later in this introduction.

In \cite{Hemmer}, Hemmer conjectured, motivated by known results on tilting modules for Schur algebras,
that the signed Young modules are exactly the self-dual
modules for symmetric groups with Specht filtrations.
This was shown to be false in \cite{PagetFiltrations};
the fourth author later proved in \cite{WildonMF} that if $n \ge 66$ and
$G$ is a subgroup of $\sym{n}$ such that the ordinary character of $M = \Ind_G^{\sym{n}}\! F$ is multiplicity free,
then every indecomposable summand of $M$ is a self-dual module with a Specht filtration. Despite
the failure of Hemmer's conjecture, it is clear that signed Young modules are of considerable interest.
In particular, a strong connection between simple Specht modules
and signed Young modules has been established by Hemmer \cite{Hemmer} and by
Danz and the second author \cite{SDanzKJLim}. More precisely, Hemmer showed that every simple Specht module is isomorphic to a signed Young module, and Danz and the second author established their labels.

In Section~\ref{S:Main} we study \textit{signed $p$-Kostka numbers}, defined to be
the multiplicities of signed Young modules as direct summands of signed Young permutation modules.
These generalize the $p$-Kostka numbers considered in \cite{FangHenkeKoenig}, \cite{CGill},
\cite{Henke} and \cite{HenkeKoenig}.
Given the $p$-Kostka numbers for $\sym{n}$ it is routine
to calculate the decomposition matrix of $\sym{n}$ in characteristic~$p$ (see \cite[\S 3]{CGill}). It is therefore
no surprise that a complete understanding of the
$p$-Kostka numbers seems to be out of reach. However, as the references above demonstrate,
many partial results and significant advances have been
obtained.
Our first main theorem is a relation
between signed $p$-Kostka numbers. We refer the reader to Notation \ref{N:Omega} for the definitions of the composition~$\mbf{\delta}_0$ and the set
$\LambdaABR$.

\begin{thm}\label{T:YmultMY} Let $(\alpha|\beta),(\lambda|p\mu)\in\P^2(n)$. Then
\[\bigl[M(p\alpha|p\beta):Y(p\lambda|p^2\mu)\bigr]\leq
\bigl[M(\alpha|\beta):Y(\lambda|p\mu)\bigr].\] Furthermore, if $\mbf{\delta}_0=\varnothing$ for all $(\mbf{\gamma}|\mbf{\delta}) \in \LambdaABR$ then equality holds.
%Furthermore, if there exists $(\mbf{\gamma}|\mbf{\delta}) \in \LambdaABR$ such that $\mbf{\delta}_0=\varnothing$ then equality holds.
\end{thm}

Example \ref{Eg:strictineq} shows that strict inequality may
hold in Theorem~\ref{T:YmultMY}. This is an important fact, since it appears to rule out
a routine proof of Theorem~\ref{T:YmultMY} using the theory of weights for the Schur superalgebra:
we explain this obstacle later in the introduction.
However, in Corollary \ref{C:asymptotic}, we obtain the following asymptotic stability of the signed $p$-Kostka numbers
\[ \bigl[M(\alpha|\beta):Y(\lambda|p\mu)\bigr]\geq
\bigl[M(p\alpha|p\beta):Y(p\lambda|p^2\mu)\bigr]=
\bigl[M(p^2\alpha|p^2\beta):Y(p^2\lambda|p^3\mu)\bigr]=\cdots.\]
If $\beta = \varnothing$ then the condition on $\mbf{\delta}_0$ holds for all
$(\mbf{\gamma}|\mbf{\delta}) \in \LambdaABR$ and
Theorem~\ref{T:YmultMY} specializes to
Gill's result \cite[Theorem 1]{CGill} that $[M^{p \alpha} : Y^{p \lambda}] = [M^\alpha : Y^\lambda]$
for all partitions~$\alpha$ and $\lambda$ of $n$.
%; here $M^\alpha = M(\alpha|\varnothing)$ is the Young
%permutation module labelled by $\alpha$.

Our second main theorem describes
the relation between signed $p$-Kostka numbers
for partitions differing by a $p$-power of a partition.
Let $\C^2(m)$ be the set consisting of all pairs of compositions
$(\alpha|\beta)$ such that $|\alpha|+|\beta|=m$. We refer the reader to Equation \ref{eq:ellp}
in Subsection~\ref{S:Apps} for the definition of $\ell_p(\lambda|p\mu)$.

\begin{thm}\label{T:2}
Let $m$, $n$ and $k$ be natural numbers.
Let $(\pi|\widetilde\pi)\in \C^2(m)$, $(\lambda|p\mu) \in\P^2(m)$, $(\phi|\widetilde\phi)\in \C^2(n)$ and $(\alpha|p\beta) \in\P^2(n)$.
If $k>\ell_p(\lambda|p\mu)$, then
\[ \begin{split}
\bigl[M(\pi+p^k\phi|\widetilde\pi+p^k\widetilde\phi):Y(\lambda+{}&{}p^k\alpha|p(\mu+p^k\beta))\bigr]
\\
& \geq \bigl[M(\pi|\widetilde\pi):Y(\lambda|p\mu)\bigr]
\bigl[M(p\phi|p\widetilde\phi):Y(p\alpha|p^2\beta)\bigr]. \end{split}
 \]
Moreover, if $p^k>\mathrm{max}\{\pi_1,\widetilde{\pi}_1\}$, then equality holds.
\end{thm}

In particular, taking $\phi = \alpha = (r)$ and $\widetilde{\phi} = \beta =
\varnothing$, we see that
$\bigl[M(\pi+p^k(r)| \widetilde{\pi}) : Y(\lambda+p^k(r)|p\mu)\bigr] \ge
M\bigl[(\pi|\widetilde{\pi}) : Y(\lambda|p\mu)\bigr]$
with equality whenever $p^k > \max\{\pi_1,\widetilde{\pi_1}\}$.

Our third main theorem
classifies the indecomposable signed Young permutation modules.

\begin{thm}\label{T:IndecSigned}
 Let $(\alpha|\beta)\in\P^2(n)$. The signed Young permutation module $M(\alpha|\beta)$ is indecomposable if and only if one of the following conditions holds.
\begin{enumerate}
  \item [(i)] $(\alpha|\beta)=((m)|(n))$ for some non-negative integers $m,n$ such that either
  \begin{enumerate}
  \item [(a)] $m=0$,
  \item [(b)] $n=0$, or
  \item [(c)] $m+n$ is divisible by $p$.
  \end{enumerate}
  \smallskip
  \item [(ii)] $(\alpha|\beta)$ is either $((kp-1,1)|\varnothing)$ or $(\varnothing|(kp-1,1))$
  for some $k \in \N$.
\end{enumerate}
In cases \emph{(i)(a)} and \emph{(i)(b)}, we have
 $\End_{F\sym{n}}M(\alpha|\beta) \cong F$. In the remaining cases we have
$\End_{F\sym{n}}M(\alpha|\beta) \cong F[x]/\langle x^2 \rangle$.
\end{thm}

In particular, Theorem~\ref{T:IndecSigned} classifies all indecomposable
Young permutation modules up to isomorphism, recovering \cite[Theorem 2]{CGill}
for fields of odd characteristic.
Note that the Young permutation module
$M^{(n-1,1)} = M\bigl( (n-1,1) | \varnothing \bigr) = M\bigl((n-1)|(1)\bigr)$ appears in both parts (i) and (ii).
If $M(\alpha|\beta)$ is indecomposable then there exist unique
partitions $\lambda$ and $\mu$ such that $M(\alpha|\beta) \cong
Y(\lambda|p\mu)$. These  partitions are determined in Proposition~\ref{prop:labels}.

\subsection*{Schur algebras}
Our results may be applied to obtain corollaries
on modules for the Schur algebra.
Fix $n$, $d \in \N$ with $d \ge n$ and let $\GL_d(F)$ be the general linear group
of $d \times d$ matrices over $F$.
Let $\rho : \GL_d(F) \rightarrow \GL_m(F)$ be a representation of $\GL_d(F)$ of dimension $m$.
We say that $\rho$ is a \emph{polynomial representation
of degree~$n$} if the matrix coefficients $\rho(X)_{ij}$
for each $i,j \in \{1,\ldots, m\}$ are polynomials of degree $n$ in the coefficients of the
matrix $X$.
Given a polynomial representation $\rho : \GL_d(F) \rightarrow \GL(V)$
of degree $n$, the image of $V$ under
the \emph{Schur functor} $f$
is the subspace of~$V$ on which
the diagonal matrices $\mathrm{diag} (a_1,\ldots,a_d) \in \GL_d(F)$ act as $a_1\ldots a_n$.
It is easily seen that $f(V)$ is preserved by the permutation matrices in $\GL_d(F)$
that fix the final $d-n$ vectors in the standard basis of~$F^d$. Thus
$f(V)$ is a module for $F\sym{n}$.

The category of polynomial representations of $\GL_d(F)$ of degree $n$ is equivalent to
the category of modules for the Schur algebra $S_F(d,n)$. We refer the reader to \cite{GreenGLn}
for the definition of $S_F(d,n)$ and further background.
In this setting, the Schur functor may be defined by $V \mapsto eV$ where
$e \in S_F(d,n)$ is an idempotent such that $eS_F(d,n)e \cong F\sym{n}$. It follows that
$f$ is an exact functor from the category of polynomial representations
of $\GL_d(F)$ of degree $n$ to the category of $F\sym{n}$-modules.

Let $E$ denote the natural $\GL_d(F)$-module. Given $\alpha \in \P(n)$,
let $\Sym^\alpha(E)$ and $\bigwedge^\beta(E)$
denote the corresponding
divided symmetric and exterior powers of~$E$, defined as quotient modules of $E^{\otimes n}$.
The \emph{mixed powers}
$\Sym^\alpha E \otimes \bigwedge^\beta E$ for $(\alpha | \beta) \in \P^2(n)$
generate the category of $\GL_d(F)$-modules of degree $n$.
In \cite{SDonkin}, Donkin defines a \emph{listing module} to be an indecomposable
direct summand of a mixed power. (As the name suggests, listing modules generalize
tilting modules).
By \cite[Proposition~3.1c]{SDonkin}, for each $(\lambda|\mu) \in \P^2(n)$
 there
exists a unique listing module $\List(\lambda|p\mu)$ such that
$f(\List(\lambda|p\mu)) \cong Y(\lambda|p\mu)$.
By \cite[Proposition~3.1a]{SDonkin}, we have
$f(\Sym^\alpha E \otimes \bigwedge^\beta E) \cong
M{(\alpha | \beta)}$. Moreover, by \cite[Proposition~3.1b]{SDonkin},
the Schur functor induces an isomorphism
\[ \End_{\GL_d(F)}\bigl(\Sym^\alpha E \otimes \textstyle{\bigwedge^\beta} E \bigr)
\cong \End_{\sym{n}}\bigl( M(\alpha|\beta)\bigr). \]
Thus
each of our three main theorems
has an immediate translation to a result
on multiplicities of listing modules in certain mixed powers.
For example Theorem~\ref{T:IndecSigned} classifies the indecomposable $\GL_d(F)$-mixed powers
and shows that each has an endomorphism
algebra, as a $\GL_d(F)$-module, of dimension at most~$2$.

\subsection*{Steinberg Tensor Product Formula}
As Gill remarks in \cite{CGill}, some of his results can be obtained using weight spaces
and the Steinberg
Tensor Product Theorem for irreducible representations of $\GL_d(F)$. We explain the connection here, since
this remark is also relevant to this work.
Let $\alpha$ be a composition of $n$ where $d \ge n$ and let $\xi_\alpha \in S_F(d,n)$ be
the idempotent defined in \cite[Section 3.2]{GreenGLn} such that
$\xi_\alpha V$ is the $\alpha$-weight space, denoted~$V_\alpha$,
of the $S_F(d,n)$-module $V$; the idempotent $e$ defining the Schur functor is $\xi_{(1^n)}$.
For $\lambda$ a partition of $n$,
let $L(\lambda)$ denote the irreducible representation of $\GL_d(F)$ with highest weight $\lambda$,
thought of as a module for $S_F(d,n)$. Let $\Proj(\lambda)$ be the projective cover of~$L(\lambda)$.
By James' original definition of Young modules (this is shown to be equivalent to ours in Subsection~\ref{S:EqvDefs}), we have $Y^\lambda = f(\Proj(\lambda))$; moreover
 %, recovering James' original definition of Young modules in
%\%cite[\S 3]{JamesYoung}; moreover,
$[M^\alpha : Y^\lambda] = [\Sym_\alpha(E) : \Proj(\lambda)] =
\dim \Hom( S(d,n)\xi_\alpha, L(\lambda) )
= \dim \xi_\alpha L(\lambda)$. % \]
(Here $\Sym_\alpha(E) \subseteq E^{\otimes n}$
is the contravariant dual, as defined in \cite[2.7a]{GreenGLn},
of  the quotient module
$\Sym^\alpha(E)$ of~$E^{\otimes n}$).
Thus
\begin{equation}\label{eq:weights} [M^\alpha : Y^\lambda] = \dim L(\lambda)_\alpha. \end{equation}

As an example of this relationship between $p$-Kostka numbers and dimensions of weight spaces,
we use~\eqref{eq:weights} to deduce Theorem 1 in \cite{CGill}.
By the Steinberg Tensor Product Theorem, $L(p\lambda) = L(\lambda)^F$, where $F$
is the Frobenius map, acting on representing matrices by sending each entry to its $p$th power.
Clearly there is a canonical vector space isomorphism $(L(\lambda)^F)_{p\alpha} \cong L(\lambda)_{\alpha}$.
Therefore
\[ [M^{p\alpha}:  Y^{p\lambda}] = \dim L(p\lambda)_{p\alpha} = \dim L(\lambda)_\alpha =
[M^\alpha : Y^\lambda] \]
as required.

\subsection*{Schur superalgebras}
Our Theorem~\ref{T:YmultMY} generalizes the result just proved, so it is natural to ask if it can
be proved in a similar way, replacing
the Schur algebra with the Schur superalgebra defined in \cite{SDonkin}.
Let $a, b \in \N$. Given $(\lambda | p\mu) \in \P^2(n)$ where $\lambda$ has at most $a$ parts
and $\mu$ has at most $b$ parts,
let $L(\lambda|p\mu)$ denote the irreducible module
of highest weight $(\lambda|p\mu)$ for the Schur superalgebra $S(a|b,n)$, defined in \cite[page 661]{SDonkin}.
By \cite[\S2.3]{SDonkin}, we have
\begin{equation}
\label{eq:weightsSuper} [M(\alpha|\beta):Y(\lambda|p\mu)] = \dim L(\lambda|p\mu)_{(\alpha|\beta)}.\end{equation}
generalizing~\eqref{eq:weights}.

Let $\GL(a|b)$ denote the super general linear group defined in \cite[Section 2]{BrundanKujawa}.
Since $E^{\otimes n}$ is a generator for the category of polynomial representations of $\GL(a|b)$
of degree $n$, it follows from~\cite[page 660, (1)]{SDonkin} that the category of such modules
is equivalent to the module category of $S(a|b,n)$.
Taking the even degree part of $\GL(a|b)$ recovers $\GL_a(F) \times \GL_b(F)$. (More precisely,
the even degree part is isomorphic to the product of the affine group schemes corresponding to these two
general linear groups). The Frobenius map is identically zero on the odd degree part of $\GL(a|b)$,
so induces a map $F : \GL(a | b) \rightarrow \GL_a(F) \times \GL_b(F)$. Let $F^\star$ be the corresponding
inflation functor, sending modules for $\GL_a(F) \times \GL_b(F)$ to modules for $\GL(a|b)$. By
\cite[Remark 4.6(iii)]{BrundanKujawa} we have
\[ L(p\lambda | p\mu) = F^\star (L(\lambda) \boxtimes L(\mu)) \]
where $\boxtimes$ denotes an outer tensor product.
Taking weight spaces we get
$L(p\lambda | p\mu)_{(p\alpha | p\beta)} \cong L(\lambda)_\alpha \boxtimes L(\mu)_{\beta}$.
By~\eqref{eq:weightsSuper} we have
$[M(p\alpha|p\beta):Y(p\lambda | p\mu)] =
\dim L(\lambda)_\alpha  \dim L(\mu)_\beta$.
Replacing $\mu$ with $p\mu$ and applying the
Steinberg Tensor Product Formula,
this implies the asymptotic stability of signed $p$-Kostka numbers
mentioned after Theorem~\ref{T:YmultMY}.

Stated for $\GL(a|b)$-modules, the remaining part of Theorem~\ref{T:YmultMY} becomes
\[ \dim L(p\lambda|p\mu^2)_{(p\alpha|p\beta)} \le \dim L(\lambda|p\mu)_{(\alpha|\beta)}. \]
This does not follow from the results mentioned so far,
or from the version of the Steinberg Tensor Product Theorem
for $\GL(a|b)$-modules proved in \cite{Kujawa}, because the module
on the right-hand side is not an inflation. Moreover, translated into this
setting, a special case of Example~\ref{Eg:strictineq} shows that
$\dim L((1)|\varnothing)_{(\varnothing|(1))} = 1$ whereas
$\dim L((p)|\varnothing)_{(\varnothing|(p))} = \dim L((p))_\varnothing \dim
L(\varnothing)_{(p)} = 0$, so it is certainly not
the case that equality always holds. (Further examples of this type are given by
the general case of Example~\ref{Eg:strictineq}).
Whether or not a proof using supergroups is possible, the authors believe
that since Theorem~\ref{T:YmultMY}
can be stated within the context of symmetric groups,
it deserves a proof in this setting.

\subsection*{Klyachko's Multiplicity Formula}
Klyachko's Multiplicity Formula \cite[Corollary 9.2]{Klyachko}
expresses the $p$-Kostka number $[M^\alpha : Y^\lambda]$
in terms of $p$-Kostka numbers for $p$-restricted partitions.
Our Corollary~\ref{cor:Klyachko} gives a generalization to signed Young modules.
%{[\bf Question: what is
%the translation of this for $\GL(a|b)$-modules?]}
Specializing this result we obtain a symmetric group proof of Klyachko's formula in the form
\begin{equation}
\label{eq:Klyachko}
[M^\alpha : Y^\lambda] = \sum_{(\mbf{\gamma}|\varnothing) \in \Lambda((\alpha|\varnothing),\rho)}
\prod_{i=0}^r [M^{\mbf{\gamma}_i} : Y^{\lambda(i)}], \end{equation}
where $\lambda = \sum_{i=0}^r p^i\lambda(i)$ is the $p$-adic expansion of $\lambda$,
$\rho$ is the partition defined in Theorem~\ref{T:Donkin}(iii) and the set
$\Lambda\bigl((\alpha|\varnothing),\rho\bigr)$ is as defined in Notation~\ref{N:Omega}.
%Stated for representations of $\GL_d(F)$, where $d$ is at least the number
%of parts of $\lambda$, this becomes
%%\begin{equation}
%%\label{eq:KlyachkoWeights}
%\[ \dim L(\lambda)_\alpha = \sum_{(\mbf{\gamma}|\varnothing) \in \Lambda((\alpha|\varnothing), \rho)}
%\prod_{i=0}^k \dim L(\lambda(i))_{\mbf{\gamma}_i} \]
%\end{equation}

%\subsection*{Kay Jin's comment}
%By \cite[\S2.3]{SDonkin}, the number $[M(\alpha|\beta):Y(\lambda|p\mu)]$ is also equal to the dimension of the $(\alpha|\beta)$-weight space of the irreducible module $L(\lambda|p\mu)$ of highest weight $(\lambda|p\mu)$ of the Schur superalgebra $S(a|b,n)$ where $a,b\geq n$, i.e. \[\dim L(\lambda|p\mu)^{(\alpha|\beta)}=[M(\alpha|\beta):Y(\lambda|p\mu)].\] Let $\F$ denote the Frobenius map $\F:\mathrm{GL}(a|b)\to \mathrm{GL}(a)\times \mathrm{GL}(b)$. Then $\F^* (L(\sigma)\boxtimes L(\tau))\cong L(p\sigma|p\tau)$ (see Remark 4.6 of Brundan-Kujawa) is the irreducible $S(a|b,n)$-module obtained by inflating through $\F$ and $L(\sigma),L(\tau)$ are the irreducible $S(a,n_1)$- and $S(b,n_2)$-modules of highest weights $\sigma,\tau$ respectively. Since $L(\sigma)\boxtimes L(\tau)$ has weights of the form $(\gamma|\delta)$, $F^*(L(\sigma)\boxtimes L(\tau))$ has weights of the form $(p\gamma|p\delta)$. Therefore, if $r\neq 0$, \[[M(\varnothing:(mp^2,rp)):Y(p(1^r):p(mp))]=\dim L(p(1^r)|p(mp))^{(\varnothing|(mp^2,rp))}=0.\] However, it is not clear that $\dim L((1^r)|p(m))^{(\varnothing|(mp,r))}=0$ when $r\neq 0$.

\subsection*{Outline}
In Section \ref{S:Prel} we recall the main ideas concerning the Brauer
construction for $p$-permutation modules and set up our notation for symmetric group modules
and modules for wreath products.
In Section~\ref{S:QuotientYoungPermutation} we find the Brou\'e quotients of signed Young
permutation modules. In Section~\ref{S:YoungConstruction}
we use these results, together with James' Submodule Theorem,
to define Young modules and signed Young modules in the
symmetric group setting. We then prove Donkin's Theorem~\ref{T:Donkin}.
We give some immediate corollaries of this theorem in Section~\ref{S:Apps}. In Sections~\ref{S:Main} and~\ref{S:Last}, we prove
Theorems \ref{T:YmultMY}, \ref{T:2} and \ref{T:IndecSigned}.

\section{Preliminaries}\label{S:Prel}

We work with left modules throughout.
For background on vertices and sources and other results from modular representation theory
we refer the reader to \cite{Alperin}.
For an account of the representation theory of the symmetric group we refer the
reader to \cite{james1978representation} or \cite{JK}, or for more recent developments, to \cite{KleshchevAMS}.

\subsection{Indecomposable summands}
Let $G$ be a finite group.
Let $M$ and $N$ be $FG$-modules. We write $N\mid M$ if $N$ is isomorphic to a direct summand of $M$.
We have already used the notation
$[M:N]$ for the number of summands in a direct sum decomposition of~$M$
that are isomorphic to the indecomposable module $N$. This multiplicity is well defined by the
Krull--Schmidt Theorem (see \cite[Section 4, Lemma~3]{Alperin}). The proof of the following lemma is easy.

\begin{lem}\label{L:trivialnormal} Let $M$ and $N$ be $FG$-modules, and let $N$ be indecomposable. Suppose
that~$H$ is a normal subgroup of $G$ acting trivially on both the modules $M$ and $N$.
Let $\overline{M}$ and~$\overline{N}$ be the corresponding
$F[G/H]$-modules. Then $[M:N]=[\overline{M}:\overline{N}]$.
\end{lem}
%\begin{proof} Clearly, any indecomposable summand $L$ of $M$ gives an indecomposable $F[G/H]$-module $\overline{L}$. Conversely, any indecomposable $F[G/H]$-module is an indecomposable $FG$-module through inflation. Hence $[M:N]=[\overline{M}:\overline{N}]$
%by the Krull--Schmidt Theorem.
%\end{proof}

\subsection{Brou\'e correspondence}

Let $G$ be a finite group.
An $FG$-module $V$ is said to be a \emph{$p$-permutation} module if for every Sylow $p$-subgroup $P$ of $G$ there exists a linear basis of~$V$ that is permuted by $P$.
A useful characterization of $p$-permutation modules is given by the following theorem
(see \cite[(0.4)]{Broueperm}).

\begin{thm}
An indecomposable $FG$-module $V$ is a $p$-permutation module if and only if there exists a $p$-subgroup $P$ of $G$ such that $V\ |\ \Ind_P^GF$; equivalently, $V$ has trivial \sourced
\end{thm}

It easily follows that the class of $p$-permutation modules is closed under
restriction and induction and under taking direct sums, direct summands and tensor products.

We now recall the definition and the basic properties of Brauer quotients. % for $FG$-modules.
Given an $FG$-module $V$ and $P$ a $p$-subgroup of $G$, the set of fixed points of $P$ on $V$
is denoted by
 \[V^P=\{v\in V\ :\ gv=v\ \text{for all}\ g\in P \}.\] It is easy to see that $V^P$ is an
$FN_G(P)$-module on which $P$ acts trivially. For $Q$ a proper subgroup of $P$,
the relative trace map $\mathrm{Tr}_Q^P:V^Q\rightarrow V^P$
is the linear map defined by $$\mathrm{Tr}_Q^P(v)=\sum_{g\in P/Q}gv,$$
where the sum is over a
complete set of left coset representatives for $Q$ in $P$.
The definition of this map does not depend on the choice of the set of representatives. We observe that
 $$\mathrm{Tr}^P(V)=\sum_{Q<P}\mathrm{Tr}_Q^P(V^Q)$$ is an $FN_G(P)$-module on which $P$ acts
 trivially. We define the \textit{Brauer quotient} of~$V$ with respect to $P$
 to be the $F[N_G(P)/P]$-module
 $$V(P)=V^P/\mathrm{Tr}^P(V).$$
If $V$ is an indecomposable $FG$-module and $P$ is a $p$-subgroup of $G$
such that  $V(P)\neq 0$, then $P$ is contained in a \vertex of
$V$. Brou\'e proved in \cite{Broueperm} that the converse holds for $p$-permutation modules.

\begin{thm}[{\cite[Theorem 3.2]{Broueperm}}]\label{BT2}
Let $V$ be an indecomposable $p$-permutation module and $P$ be a \vertex of $V$.
Let $Q$ be a $p$-subgroup of $G$. Then $V(Q)\neq 0$ if and only if $Q\leq {}^gP$ for some $g\in G$.
\end{thm}

Here ${}^gP$ denotes the conjugate $gPg^{-1}$ of $P$.
If $V$ is an $FG$-module with $p$-permutation basis $\mathcal{B}$ with respect
to a Sylow $p$-subgroup $\widetilde{P}$ of $G$ and $P \le \widetilde{P}$, then,
taking for each orbit of $P$ on $\mathcal{B}$ the sum
of the basis elements in that orbit, we obtain a basis for $V^P$.
Each sum over an orbit of size $p$ or more
is a relative trace from a proper subgroup of~$P$. Hence $V(P)$ is isomorphic to the $F$-span of
$$\mathcal{B}^P = \{ v \in \mathcal{B} : \text{$gv = v$ for all $g\in P$} \}.$$
Thus Theorem~\ref{BT2} has the following
corollary.

\begin{cor}\label{cor:Brauer}
Let $V$ be a $p$-permutation $FG$-module
with
$p$-permutation basis
$\mathcal{B}$ with respect to a Sylow $p$-subgroup of $G$ containing a subgroup $P$. % $P \le \widetilde{P}$.
The Brauer quotient $V(P)$ has~$\mathcal{B}^P$ as a basis. %. %_F$
Moreover, $V$ has an indecomposable summand with a
\vertex containing~$P$ if and only if $\mathcal{B}^P \not= \emptyset$.
\end{cor}

The next result states what is now known as the  Brou\'e correspondence.

\begin{thm}[{\cite[Theorems 3.2 and 3.4]{Broueperm}}]\label{BC1}
An indecomposable $p$-permutation $FG$-module~$V$ has \vertex $P$ if and only if $V(P)$ is a
projective $F[N_G(P)/P]$-module. Furthermore, we have the following statements.
\begin{enumerate}
\item [(i)] The Brauer map sending $V$ to $V(P)$ is a
bijection between the isomorphism classes of indecomposable $p$-permutation
$FG$-modules with \vertex $P$ and the isomorphism classes of indecomposable
projective modules for $F[N_G(P)/P]$. Regarded as an $FN_{G}(P)$-module, $V(P)$ is the
Green correspondent of~$V$.
\item [(ii)] Suppose that $V$ has vertex $P$.
If $M$ is a $p$-permutation $FG$-module then~$V$ is a direct summand of $M$ if and only if $V(P)$ is a direct summand of $M(P)$. Moreover, \[[M:V]=[M(P):V(P)].\]
\end{enumerate}
\end{thm}

The following lemma allows the Brou{\'e} correspondence to be applied to monomial modules
such as signed Young permutation modules.

\begin{lem}\label{L:linearperm} Let $A$ be a subset of $F^\times$. Let $M$ be an $FG$-module with an $F$-basis $\mathcal{B} = \{m_1,\ldots,m_r\}$ such that, if $g\in G$ and $m_i \in \mathcal{B}$ then
 $gm_i=am_j$ for some~\hbox{$a \in A$} and some $m_j \in \mathcal{B}$.
 Then, for any $p$-subgroup $P$ of $G$, there exist coefficients $a_1,\ldots,a_r\in A$ such that
 $\{a_1m_1,\ldots,a_rm_r\}$ is a $p$-permutation basis of $M$ with respect to $P$.
\end{lem}
\begin{proof} Let $\{i_1,\ldots,i_s\}$ be a subset of $\{1,\ldots,r\}$ such that $\mathcal{B}$ is the disjoint union of $\mathcal{B}_1,\ldots,\mathcal{B}_s$ where, for each $j$,
\[\mathcal{B}_j=\{m_k\ :\ \text{$gm_{i_j}=a_gm_k$ for some $g\in P$ and $a_g\in A$}\}.\] Suppose that $gm_{i_j}=am_k$ and $g'm_{i_j}=a'm_k$ for some $g,g'\in P$ and $a,a'\in A$. Then $g^{-1}g'm_{i_j}=a'a^{-1}m_{i_j}$ and, consequently, $Fm_{i_j}$ is a one dimensional $F\langle g^{-1}g'\rangle$-module. Since $P$ is a $p$-subgroup,  $Fm_{i_j}$ is the trivial $F\langle g^{-1}g'\rangle$-module. Hence $a=a'$.
Thus the coefficient $a_g$ is independent of the choice of $g$, and depends  only on $m_{i_j}$ and~$m_k$.

For each $1\leq j\leq s$, let \[\Lambda_j=\{a_km_k\ :\ \text{$gm_{i_j}=a_km_k$ for some $g\in P$ and $a_k\in A$}\}.\]
By the previous paragraph, $\bigcup_{j=1}^s \Lambda_j$ is a basis of $M$.
It is sufficient to prove that each
$\Lambda_j$ is permuted by $P$. Let $x\in P$, and let $a_km_k$, $a_{k'}m_{k'}\in\Lambda_j$.
Suppose that $x(a_{k'}m_{k'})=b(a_km_k)$ for some $b\in F$. We have
%\marginpar{KJL: $b$ may not lie in $A$}
$gm_{i_j}=a_km_k$ and $g'm_{i_j}=a_{k'}m_{k'}$ for some $g$, $g'\in P$. Thus $g^{-1}xg'm_{i_j}=bm_{i_j}$.
Repeating the argument in the first paragraph,
we see that $Fm_{i_j}$ is the trivial $F\langle g^{-1}xg'\rangle$-module
and so $b=1$.
\end{proof}

The Brauer quotient of an outer tensor product of $p$-permutation modules is easily described.

\begin{lem}\label{L:boxtimesBroue} Let $G_1$ and $G_2$ be finite groups,
let $M_1,M_2$ be $p$-permutation $FG_1$- and $FG_2$-modules, and let $P_1$ and $P_2$ be $p$-subgroups of $G_1$, $G_2$, respectively. Then \[(M_1\boxtimes M_2)(P_1\times P_2)\cong M_1(P_1)\boxtimes M_2(P_2)\]
as a representation of
$N_{G_1\times G_2}(P_1\times P_2)/(P_1\times P_2) \cong (N_{G_1}(P_1)/P_1)\times (N_{G_2}(P_2)/P_2)$.
\end{lem}
\begin{proof}
The statement follows from an easy application of Theorem \ref{BC1}.
%Let $\mathcal{B}_1$, $\mathcal{B}_2$ be permutation bases for $M_1$, $M_2$ with respect to the $p$-subgroups $P_1$,~$P_2$ respectively. Then $M_1\boxtimes M_2$ has permutation basis $\mathcal{B}_1\otimes \mathcal{B}_2 =
%\{ m_1 \otimes m_2 : m_1 \in \mathcal{B}_1, m_2 \in \mathcal{B}_2 \}$ with respect to the subgroup $P_1\times P_2$. So $(M_1\boxtimes M_2)(P_1\times P_2)$ has a basis
%$(\mathcal{B}_1\otimes \mathcal{B}_2)^{P_1 \times P_2} %(P_1\times P_2) %=\mathcal{B}_1(P_1)\otimes \mathcal{B}_2(P_2)
%=
%$\{m_1 \otimes m_2 :
%m_1 \in \mathcal{B}_1(P_1), m_2 \in \mathcal{B}_2(P_2)\}$. The lemma now follows
%from Corollary~\ref{cor:Brauer}.
\end{proof}

\subsection{Partitions and compositions}\label{Sec:PCs}
Let $n\in \N_0$. A \emph{composition} of $n$ is a sequence of non-negative integers $\alpha=(\alpha_1,\ldots,\alpha_r)$ such that $\alpha_r\neq 0$ and $\alpha_1+\cdots+\alpha_r=n$ . In this case, we write $\ell(\alpha)=r$ and $|\alpha|=n$. The unique composition of $0$ is denoted by $\varnothing$; we have $\ell(\varnothing)=0$. The \emph{Young subgroup} $\sym{\alpha}$ is the subgroup $\sym{\alpha_1}\times\cdots\times\sym{\alpha_r}$ of $\sym{n}$, where the $i$th factor $\sym{\alpha_i}$ acts on the set $\{\alpha_1+\cdots+\alpha_{i-1}+1,\ldots,\alpha_1+\cdots+\alpha_{i-1}+\alpha_i\}$.
Let $\alpha=(\alpha_1,\ldots,\alpha_r)$ and $\beta=(\beta_1,\ldots,\beta_s)$ be compositions and let
$q\in \N$. We denote by $q \pdot \alpha$ and $\alpha\bullet\beta$ the compositions of $q|\alpha|$ and $|\alpha|+|\beta|$  defined by
\begin{align*}
q\alpha =& (q\alpha_1,\ldots,q\alpha_r),\\
%\alpha+\beta =& (\alpha_1+\beta_1,\alpha_2+\beta_2,\ldots),\\
\alpha\bullet\beta=&(\alpha_1,\ldots,\alpha_r,\beta_1,\ldots,\beta_s),
\end{align*}
respectively. We set $0\alpha=\varnothing$.
We denote by $\alpha+\beta$ the composition of $|\alpha|+|\beta|$ defined by \[\alpha+\beta=(\alpha_1+\beta_1,\ldots,\alpha_r+\beta_r,\alpha_{r+1},\ldots,\alpha_s)\]
where we have assume, without loss of generality, that  $r\leq s$. We define $\alpha-\beta$ similarly,
in the case when $\beta_i \le \alpha_i$ for each $i \le \ell(\beta)$.

A composition $\alpha$ is a \emph{partition} if it is non-increasing.
A partition $\alpha$ is \emph{$p$-restricted} if $\alpha_i-\alpha_{i+1}<p$ for all $i\geq 1$. We denote the set of compositions, partitions and $p$-restricted partitions of $n$ by $\C(n)$, $\P(n)$ and $\RP(n)$, respectively. A partition~$\alpha$ is \emph{$p$-regular} if its conjugate $\alpha'$,
defined by $\alpha'_j = |\{ i : \alpha_i \ge j \}|$, is $p$-restricted.
It is well known that if $\lambda$ is a partition then there exist unique $p$-restricted
partitions $\lambda(i)$ for $i \in \N_0$ such that
\begin{equation}\label{eq:padic}
\lambda = \sum_{i\ge 0} p^i \pdot \lambda(i).
\end{equation}
We call this expression the \emph{$p$-adic expansion} of $\lambda$.

Let $\P^2(n)$, $\C^2(n)$ and $\RP^2(n)$ be the sets consisting of all pairs $(\lambda|\nu)$ of partitions, compositions and $p$-restricted partitions, respectively,  such that $|\lambda|+|\nu|=n$. Here $\lambda$ or $\nu$ may be the empty composition $\varnothing$. For $(\lambda|\nu), (\alpha|\beta)\in\P^2(n)$, we say that $(\lambda|\nu)$ \emph{dominates} $(\alpha|\beta)$, and write $(\lambda|\nu)\unrhd (\alpha|\beta)$, if, for all $k\geq 1$, we have
\begin{enumerate}
\item [(a)] $\sum^k_{i=1}\lambda_i\geq \sum^k_{i=1}\alpha_i$, and

\item [(b)] $|\lambda|+\sum^k_{i=1}\nu_i\geq |\alpha|+\sum^k_{i=1}\beta_i$.
\end{enumerate}
(As a standing convention we declare that $\lambda_i = 0$ whenever $\lambda$ is a partition and
$i > \ell(\lambda)$).
%If $i > \ell(\lambda)$ then take $\lambda_i=0$, and similarly for $\nu$, $\alpha$ and $\beta$).
This defines a partial order on the set $\P^2(n)$ called the \emph{dominance order}.
This order becomes the usual dominance order on partitions when restricted to the subsets $\{(\lambda|\varnothing)\in \P^2(n) \}$ %\ :\ \nu=\varnothing\}$
or $\{(\varnothing|\nu)\in \P^2(n)\}$ %\ :\ \lambda=\varnothing\}$
of $\P^2(n)$.

\subsection{Modules for symmetric groups}\label{Sec:SymModules}
Let $n\in\N_0$, let $\sym{n}$ be the symmetric group on the set $\{1,\ldots,n\}$ and let $\alt{n}$ be its alternating subgroup.
Given a subgroup $H$ of~$\sym{n}$, we denote the trivial representation
of $H$ by $F(H)$, and the restriction of the sign representation of $\sym{n}$ to $H$ by $\sgn(H)$.
In the case when $H=\sym{\gamma}$ for some composition $\gamma$ of $n$
we write $F(\gamma)$ and $\sgn(\gamma)$ for $F(H)$ and $\sgn(H)$, respectively. If $\gamma = (n)$
we reduce the number of parentheses  by writing $F(n)$ and $\sgn(n)$, respectively.

For $\lambda$ a $p$-regular partition of $n$, let $D^\lambda$ be the $F\sym{n}$-module defined
by $$D^\lambda=S^\lambda/\mathrm{rad}(S^\lambda),$$
where $S^\lambda$ is the Specht module labelled by $\lambda$ (see \cite[Chapter~4]{james1978representation}). By \cite[Theorem~11.5]{james1978representation} each $D^\lambda$
is simple, and each  simple $F\sym{n}$-module is isomorphic to a unique~$D^\lambda$.
The simple $F\sym{n}$-modules can also be labelled by $p$-restricted partitions.
For $\lambda\in\RP(n)$ we set $D_\lambda=\mathrm{soc}(S^\lambda)$.
The connection between the two labelings is given by
$$D_\lambda\cong D^{\lambda'}\otimes \sgn(n).$$
For $\lambda \in \RP(n)$, let $\Plambda$ denote
the projective cover of the simple $F\sym{n}$-module~$D_\lambda$.

Finally, for $\gamma\in\P(n)$, let $\chi^\gamma$ denote the ordinary irreducible character of $S^\gamma$, defined over
the rational field.

\subsection{Modules for wreath products}\label{Sec:Modules}
Let $m\in\N$ and let $G$ be a finite group. Recall that the multiplication in the
group $G\wr \sym{m}$ is given by
\[(g_1,\ldots,g_m;\sigma)(g_1',\ldots,g_m';\sigma')=(g_1g_{\sigma^{-1}(1)}',\ldots,g_mg_{\sigma^{-1}(m)}';\sigma\sigma')\,,\]
for $(g_1,\ldots,g_m;\sigma), \, (g_1',\ldots,g_m';\sigma')\in G\wr \sym{m}$.
(Our notation for wreath products is taken from \cite[Section 4.1]{JK}).
Let $M$ be an $FG$-module. The $m$-fold tensor product of $M$
becomes an $F[G\wr\sym{m}]$-module with the action given by
\[(g_1,\ldots,g_m;\sigma)\cdot (v_1\otimes\cdots\otimes v_m)=\sgn(\sigma)g_1v_{\sigma^{-1}(1)}\otimes\cdots\otimes g_mv_{\sigma^{-1}(m)}\] for  $(g_1,\ldots,g_m;\sigma)\in G\wr\sym{m}$ and
$v_1,\ldots,v_m\in M$. We denote this module by~$\widehat{M}^{\otimes m}$.
Note that we have twisted the action of the top group $\sym{m}$ by the sign representation.
Thus, in the notation of \cite[4.3.14]{JK}, we have
\[ \widehat{M}^{\otimes m}
= (\stackrel{m}{\#} M)^{\,\widetilde{}} \otimes \Inf_{\sym{m}}^{G \wr \sym{m}} (\sgn(m)). \]

The $1$-dimensional module $\widehat{\sgn(k)}^{\otimes n}$ will be important to us.
In our applications $k$ will be a $p$-power, and so odd.
Since a
transposition in the top group $\sym{n}$ acts on $\{1,\ldots, kn\}$ as a product
of $k$ disjoint transpositions, and so has odd sign, there is a simpler definition
of this module, as $\Res^{\sym{kn}}_{\sym{k} \wr \sym{n}} \sgn(kn)$. More generally,
given $\alpha \in \C(n)$ and an odd number~$k$, we define
\begin{equation}\label{eq:sgnres} \widehat{\sgn(k)}^{\otimes \alpha} = \Res^{\sym{kn}}_{\sym{k} \wr
\sym{\alpha_1} \times \cdots \times \sym{k} \wr \sym{\alpha_r}} \sgn(kn). \end{equation}

%\subsection{Characters of $C_2 \wr \sym{n}$}\label{Sec:TypeB}

For use in the proof of Proposition~\ref{P:dominates} we briefly recall the character theory
of the group $C_2 \wr \sym{n}$. Let $\chi^\lambda$ denote
the irreducible character of $\sym{n}$ labelled by $\lambda \in \sym{n}$.
For
$(\lambda|\mu) \in \P^2(n)$, with $|\lambda|= m_1$ and $|\mu| = m_2$,
we define $\chi^{(\lambda|\mu)}$ to be the ordinary character of the following module for $C_2\wr\sym{n}$
\[ \Ind_{C_2\wr(\sym{m_1}\times\sym{m_2})}^{C_2\wr\sym{n}}\left
(\Inf^{C_2\wr\sym{m_1}}_{\sym{m_1}} (\chi^\lambda) \boxtimes
(\Inf^{C_2\wr\sym{m_2}}_{\sym{m_2}}(\chi^\mu)\otimes \widehat{\sgn(2)}^{\otimes m_2})\right ).\]
A standard Clifford theory argument (see for instance \cite[Theorem 4.34]{JK})
shows that the characters $\chi^{(\lambda|\mu)}$ for $(\lambda|\mu) \in \P^2(n)$
are precisely the irreducible characters of $C_2 \wr \sym{n}$.

\subsection{Sylow $p$-subgroups of $\sym{n}$}\label{Sec:Syl}

 Let $P_p$ be the cyclic group $\langle (1,2,\ldots,p)\rangle\leq \sym{p}$ of order $p$. Let
$P_1=\{1\}$ and, for $d\geq 1$, set
$$P_{p^{d+1}}=P_{p^d}\wr P_p =\{(\sigma_1,\ldots,\sigma_p;\pi): \sigma_1,\ldots,\sigma_p\in P_{p^d},\, \pi\in P_p\}\,.$$
By \cite[4.1.22, 4.1.24]{JK},  $P_{p^d}$ is a Sylow $p$-subgroup of $\sym{p^d}$.

Let $n \in \N$. Let $n=\sum_{i=0}^r n_ip^i$, where
$0\leq n_i < p$ for $i\in\{0,\ldots, r\}$, and $n_r \not=0$ be the $p$-adic expansion of $n$.
By \cite[4.1.22, 4.1.24]{JK}, the Sylow $p$-subgroups of $\mathfrak{S}_n$ are each
conjugate to the direct product $\prod_{i=0}^r(P_{p^i})^{n_i}$. Hence
if we define $P_n$ to be a Sylow $p$-subgroup of the Young subgroup $\prod_{i=0}^r(\sym{p^i})^{n_i}$
then $P_n$ is a Sylow $p$-subgroup of $\sym{n}$.
The normalizer $N_{\sym{n}}(P_n)$ of $P_n$ in $\sym{n}$ is denoted by $N_n$.

Whenever $\rho=(\rho_1,\ldots,\rho_r)\in\mathscr{C}(n)$, we denote by $P_\rho$ a
Sylow $p$-subgroup of $\sym{\rho}$, defined so that $P_\rho=\prod_{i=1}^rP_{\rho_i}$.
In the special case when
\begin{align*}
\rho& =(1^{m_0},p^{m_1},\ldots,(p^s)^{m_s})=(\underbrace{1,\ldots,1}_{\text{$m_0$ copies}},\,\underbrace{p,\ldots,p}_{\text{$m_1$ copies}},\,\ldots,\,\underbrace{p^s,\ldots,p^s}_{\text{$m_s$ copies}}),
\end{align*}
where $m_i \in \N_0$ for each $i$,
we have $P_\rho=\prod_{i=0}^s(P_{p^i})^{m_i}$; in particular,
the group $P_\rho$ has precisely
$m_i$ orbits of size $p^i$ on the set $\{1,2,\ldots,n\}$ for each $i$. We write $N_\rho = N_{\sym{n}}(P_\rho)$.

\section{The Brauer quotients of signed Young permutation modules}\label{S:QuotientYoungPermutation}

In this section, we determine the Brauer quotients of signed Young permutation modules with respect
to Sylow subgroups of Young subgroups. Our main result is Proposition~\ref{P:Broueperm}; this
generalizes \cite[Proposition 1]{ErdmannYoung}.
The description of the Brauer quotients is combinatorial, using
the $(\alpha|\beta)$-tableaux defined below.

%As a special case (see Corollary~\ref{cor:YoungBroue})
%we obtain Proposition 2.9 in \cite{ErdSchr}; the reader interested mainly
%in the construction of Young modules may read the statement of this corollary and then
%skip to Section~\ref{S:YoungConstruction}.

Fix $n\in \N$ and $(\alpha|\beta)\in\C^2(n)$. Suppose that $\alpha=(\alpha_1,\ldots,\alpha_r)$ and $\beta=(\beta_1,\ldots,\beta_s)$. % where $r=\ell(\alpha)$ and $s=\ell(\beta)$.
The \emph{diagram} $[\alpha]\bbullet[\beta]$ is the set consisting of the \emph{boxes}
$(i,j)\in\N^2$ for $i$ and $j$ such that \emph{either}
$1\leq i\leq r$ and $1 \le j \le \alpha_i$ \emph{or} $r+1 \le i \le r+s$ and $1 \le j \le \beta_{i-r}$.
A box $(i,j)$ is said to be in \emph{row} $i$.
The subset of $[\alpha]\bbullet[\beta]$ consisting of the boxes
belonging to the first $r$ rows (respectively, the last $s$ rows) is denoted by $[\alpha]\bbullet\varnothing$ (respectively,~$\varnothing\bbullet[\beta]$).

\begin{defn}\label{defn:pairTableau}
An $(\alpha|\beta)$-\emph{tableau} $\T$ is
a bijective function $\T:[\alpha] \bbullet [\beta] \rightarrow \{1,\dots,n\}$.
For $(i,j) \in [\alpha]\bbullet [\beta]$, the $(i,j)$-\emph{entry} of $\T$ is $\T(i,j)$.
\end{defn}

We represent an \emph{$(\alpha|\beta)$-tableau} $\T$ by putting the $(i,j)$-entry of $\T$ in
 the box $(i,j)$ of the diagram $[\alpha]\bbullet[\beta]$.
Considering $[\alpha]\bbullet\varnothing$ as the Young diagram $[\alpha]$, we denote the $\alpha$-tableau $\T([\alpha]\bbullet\varnothing)$ by $\T_+$. Similarly, we denote the $\beta$-tableau $\T(\varnothing\bbullet[\beta])$ by $\T_-$. It will sometimes be useful to write $$\T=(\T_+|\T_-).$$
%This combinatorial interpretation of a tableau $\T$ will be particularly important in the proof of Theorem %\ref{T:2}.
The $(\alpha|\beta)$-tableau $\T$ is \emph{row standard} if the entries in each row of $\T$ are increasing from left to right, i.e. both $\T_+$ and $\T_-$ are row standard in the usual sense.
We denote by $\T^{\alpha|\beta}$ the unique row standard $(\alpha|\beta)$-tableau
such that
for all $i$, $j \in \{1,\ldots,n \}$, if $i$ is in row $a$ of $\T^{\alpha|\beta}$ and $j$ is in row $b$ of $\T^{\alpha|\beta}$
and $i \le j$ then $a \le b$. For example,
\[ \T^{(2,1) | (3)} = \doubleyoungtableauonerow(12,3)(456) \]
where the thicker line separates the two parts of the tableau.

Let $\TT(\alpha|\beta)$ be the set of all $(\alpha | \beta)$-tableaux.
If $\T\in\TT(\alpha|\beta)$ and $g\in \sym{n}$ then we define $g\cdot \T$ to be
the $(\alpha|\beta)$-tableau obtained by applying $g$ to each entry of $\T$, i.e. $(g\cdot \T)(i,j)=g(\T(i,j))$. This defines an action of $\sym{n}$ on the set $\TT(\alpha|\beta)$.
The vector space
$F\TT(\alpha|\beta)$ over $F$ with basis $\TT(\alpha|\beta)$ is therefore
a permutation $F\sym{n}$-module.

For each $\T\in\TT(\alpha|\beta)$, let $R(\T)\leq \sym{n}$ be the {\sl row stabilizer} of $\T$ in $\sym{n}$, consisting of those $g\in\sym{n}$ such that the rows of $\T$ and $g\cdot \T$ coincide as sets. Then
$R(\T)=R(\T_+)\times R(\T_-)$, where $R(\T_+)$ and $R(\T_-)$ are the row stabilizers of $\T_+$ and $\T_-$ respectively, in the usual sense.
Denote by $U(\alpha|\beta)$ the subspace of $F\TT(\alpha|\beta)$ spanned by
$$\{\T-\sgn(g_2) g_1g_2\cdot \T: \T\in \TT(\alpha|\beta),\, (g_1,g_2)\in R(\T_+)\times R(\T_-)\}.$$ In fact $U(\alpha|\beta)$ is an $F\sym{n}$-submodule of $F\TT(\alpha|\beta)$, since for all $h\in \sym{n}$ and for any
$(g_1,g_2)\in R(\T_+)\times R(\T_-)$ and $\T\in\TT(\alpha|\beta)$ we have %, and $g=g_1g_2$, we have
$$h\cdot(\T-\sgn(g_2) g\cdot \T)=h\cdot \T-\sgn({}^h\negthinspace g_2) {}^h\negthinspace g\cdot (h\cdot \T)
\in U(\alpha|\beta),$$
where $g = g_1g_2$,
since  ${}^h\negthinspace g\in {}^h\negthinspace R(\T)=R(h\cdot \T)$ and ${}^h\negthinspace g_2\in R((h\cdot \T)_-)$.
%Here and thereafter, ${}^hg$ denotes the conjugation $hgh^{-1}$.

\begin{defn}\label{def:tabloid}
For each $\T\in\TT(\alpha|\beta)$, we write \[\{\T\}=\{(\T_+|\T_-)\}\] for the element $\T+U(\alpha|\beta)\in F\TT(\alpha|\beta)/U(\alpha|\beta)$ and
call it an \textit{$(\alpha|\beta)$-tabloid}.
\end{defn}

Note that $g\{\T\}=\{g\cdot \T\}$ for all $g\in\sym{n}$ and $\T\in\TT(\alpha|\beta)$.
If $\T,\T'\in \TT(\alpha|\beta)$ are such that $\T_-=\T'_-$ and
$\T'_+$ is obtained by swapping two entries in the same row of $\T_+$ then $\{\T\}=\{\T'\}$. On the other hand, if $\T_+=\T'_+$ and $\T'_-$ is obtained by swapping two entries in the same row of $\T_-$ then $\{\T'\}=-\{\T\}$.
The graphical representation of $(\alpha|\beta)$-tableaux is shown
in Example~\ref{eg:tabloids} below.

Let \[\Omega(\alpha|\beta)=\bigl\{\{\T\}:\text{$\T$ is a row standard $(\alpha|\beta)$-tableau}\bigr\}
\subseteq F\TT(\alpha|\beta)/U(\alpha|\beta).\] It is clear that $\Omega(\alpha|\beta)$ is an $F$-basis of $F\TT(\alpha|\beta)/U(\alpha|\beta)$. We write $F\Omega(\alpha|\beta)$ for the $F\sym{n}$-module $F\TT(\alpha|\beta)/U(\alpha|\beta)$.

\begin{lem}\label{L:signpermbasis} Let $(\alpha|\beta)\in\C^2(n)$.% and let $\Omega(\alpha|\beta)$ be the basis of $F\overline{\TT}(\alpha|\beta)$ consisting of $(\alpha|\beta)$-tabloids. Then
\begin{enumerate}
\item [(i)] The $F\sym{n}$-module $F\Omega(\alpha|\beta)$  is isomorphic to the signed Young permutation $M(\alpha|\beta)$.
\item [(ii)] For any $p$-subgroup $P$ of $\sym{n}$, there exist coefficients $a_{\{\T\}}\in\{\pm 1\}$ for
 each $\{\T\}\in\Omega(\alpha|\beta)$, such that
$$\bigl\{a_{\{\T\}}\{\T\}\ :\ \{\T\}\in \Omega(\alpha|\beta)\}$$ is a $p$-permutation basis for $F\Omega(\alpha|\beta)\cong M(\alpha|\beta)$ with respect to $P$.
\end{enumerate}
\end{lem}
\begin{proof}
By the remarks after Definition~\ref{def:tabloid}
there is an isomorphism $F(\alpha)\boxtimes \sgn(\beta)\cong F\{\T^{\alpha|\beta}\}$
of $F[\sym{\alpha}\times\sym{\beta}]$-modules.
Since $|\Omega(\alpha|\beta)|=\dim_FM(\alpha|\beta)$, part~(i)
follows from the characterization of induced modules in
\cite[Section 8, Corollary 3]{Alperin}.
Part~(ii) follows from Lemma \ref{L:linearperm}, since,
for all $\{\T\}\in\Omega(\alpha|\beta)$ and $\sigma\in\sym{n}$, we have
$\sigma\{\T\}=\pm\{\T'\}$ for some $\{\T'\}\in\Omega(\alpha|\beta)$.
\end{proof}

In view of Lemma \ref{L:signpermbasis}(i), we shall identify $M(\alpha|\beta)$ with
$F\Omega(\alpha|\beta)$, so that $M(\alpha|\beta)$ has the set
of $(\alpha|\beta)$-tabloids as a basis.

The next corollary follows from Lemma \ref{L:signpermbasis} and Corollary \ref{cor:Brauer}.

\begin{cor}\label{C:basisBroueM} Let $(\alpha|\beta)\in\C^2(n)$.
\begin{enumerate}
\item[(i)]  Let $P$ be a $p$-subgroup of $\sym{n}$.
The  $F[N_{\sym{n}}(P)/P]$-module $M(\alpha|\beta)(P)$ has a linear basis consisting of all the $(\alpha|\beta)$-tabloids $\{\T\}$ that are fixed by $P$.
\item[(ii)] Let $\rho=(1^{n_0},p^{n_1},\ldots,(p^r)^{n_r})$ be a partition of $n$.
%Let $N_\rho = N_{\sym{n}}(P_\rho)$.
The group \[N_\rho/P_\rho \cong \sym{n_0}\times ((N_p/P_p)\wr\sym{n_1})\times\cdots\times ((N_{p^r}/P_{p^r})\wr\sym{n_r})\] acts on the set of $P_\rho$-fixed $(\alpha|\beta)$-tabloids by transitively permuting  the
entries in $P_\rho$-orbits of size $p^i$ according to $\sym{n_i}$ and, within each $P_\rho$-orbit of size $p^i$, permuting its entries according to $N_{p^i}/P_{p^i}$, for all $i\in\{0,1,\ldots, r\}$.
\end{enumerate}
\end{cor}

More explicitly, the basis in Corollary \ref{C:basisBroueM}(i)
consists of all $(\alpha|\beta)$-tabloids $\{\T\}$ such that $\T$ is
row standard and each row of~$\T$ is a union of orbits of $P$ on $\{1,\ldots,n\}$.
This can be seen in the following example.

\begin{eg}\label{eg:tabloids} Let $p=3$. Consider the $3$-subgroups $Q_1=\langle (1,2,3),(4,5,6), (7,8,9)\rangle$ and $Q_2=\langle (4,5,6), (7,8,9)\rangle$ and  of $\sym{9}$. By Corollary~\ref{C:basisBroueM}(i), since there are no $\bigl((2,1)|(6)\bigr)$-tabloids fixed by $Q_1$, we have $M\bigl((2,1)|(6)\bigr)(Q_1)=0$. On the other hand,
$M\bigl((2,1)|(6)\bigr)(Q_2)$ has a basis consisting of the $\bigl((2,1)|(6)\bigr)$-tabloids
%\begin{align*}
%&\left \{\begin{array}{cc|ccc} 1&2&4&5&6\\ 3&&\end{array}\right \}, &
%&\left \{\begin{array}{cc|ccc} 1&3&4&5&6\\ 2&&\end{array}\right \}, &
%\left \{\begin{array}{cc|ccc} 2&3&4&5&6\\ 1&&\end{array}\right \}
%\end{align*}
\setlength{\arrayrulewidth}{.1em}
\[
\left\{ \, \doubleyoungtabloidonerow(12,3)(456789)\, \right\},\quad
\left\{ \, \doubleyoungtabloidonerow(13,2)(456789) \, \right\},\quad
\left\{ \, \doubleyoungtabloidonerow(23,1)(456789) \, \right\}
\]
where the bold line separates each $\T_+$ from $\T_-$.
Taking $\rho = (1,1,1,3,3)$ we have $P_\rho=Q_2$ and
\[ N_{\sym{9}}(Q_2) = \sym{3} \times N_{\sym{3}}(P_3) \wr \sym{2} = \sym{3} \times (\sym{3} \wr \sym{2}).\]
% \[N_{\sym{9}}(P_1)/P_1=(\sym{3}\times (N_{\sym{3}}(P_3)\wr \sym{2}))/\langle (4,5,6), (7,8,9)\rangle=\sym{3}\times \sym{3} \wr \sym{2}/\langle(4,5,6), (7,8,9)\rangle.\]
 The first factor $\sym{3}$ permutes the entries $1,2,3$ of each tabloid without sign,
 and the second factor $\sym{3}
 \wr \sym{2}$ permutes the entries $4,5,6,7,8,9$  with
 sign. The subgroup $Q_2$ acts trivially on the tabloids.
 Thus if $\{ \mathrm{U} \}$ and $\{ \mathrm{V} \}$ are the first two $((2,1)|(6))$-tabloids
 above then
 \[ \Res^{N_{\sym{9}}(Q_2)}_{\sym{3} \wr \sym{2}} (F \{ \mathrm{U}
 \}) \cong \widehat{\sgn(3)}^{\otimes 2} \]
 and
 %$\{ \T \}$ spans a representation of $\sym{3} \wr \sym{2}$ isomorphic to
 %$\widehat{\sgn(3)}^{\otimes 2}$ and
  $(23)(45) \{ \mathrm{V} \}= - \{ \mathrm{U} \}$.
Note that the isomorphism above requires the sign twist in the definition of $\widehat{M}$
for $M$ an $F\sym{m}$-module
that we commented on in Section~\ref{Sec:Modules}.
\end{eg}

%$\sigma g_i=\epsilon_1g_{i_1}$ and $\sigma g_{i_{j-1}}=\epsilon_j g_{i_j}$ for all $2\leq j\leq p^k$, where $\epsilon_1,\ldots,\epsilon_{p^k}\in \{\pm1\}$. Let $x=g_i+g_{i_1}Since $P$ is a $p$-group, we have $\sig

Given $m \in \N$, we define a $1$-dimensional $F[N_k \wr \sym{m}]$-module by
\[ \widehat{\sgn(N_k)}^{\otimes m} = \Res^{\sym{k} \wr \sym{m}}_{N_k \wr \sym{m}} \widehat{\sgn(k)}^{\otimes m}. \]
Using this we may now define three key families of modules.

%The modules
%$V_k(\gamma|\delta)$, $W_k(\gamma|\delta)$ and $\overline{W}_k(\gamma|\delta)$ in Definition \ref{D:VW}
%are the analogues of the modules $R_k(\alpha|\beta)$, $Q_k(\alpha|\beta)$ and $\overline{Q}_k(\alpha|\beta)$ defined in Definitions~\ref{def:Rk},~\ref{def:Qk} and~\ref{def:Qbark}.
%In particular, the modules $W_k(\gamma|\delta)$ will be the building blocks for the Brauer quotients of signed Young permutation modules (see Proposition \ref{P:Broueperm}).
%We leave the similar argument to the reader and summarize our definition as follows.

\begin{defn}\label{D:VW} Let $k \in \N$, let $m \in \N_0$ and let $(\gamma|\delta)\in\C^2(m)$.
Let $|\gamma|=m_1$ and $|\delta|=m_2$.
\begin{enumerate}
\item [(i)] We define  $V_k(\gamma|\delta)$ to be the $F[\sym{k} \wr \sym{m}]$-module
\[V_k(\gamma|\delta)=\Ind_{\sym{k}\wr(\sym{m_1}\times\sym{m_2})}^{\sym{k}\wr\sym{m}}
\left (\Inf^{\sym{k}\wr\sym{m_1}}_{\sym{m_1}}(M^\gamma)\boxtimes \bigl(\Inf^{\sym{k}\wr\sym{m_2}}_{\sym{m_2}}
(M^\delta)\otimes \widehat{\sgn(k)}^{\otimes m_2}\bigr)\right ).\]
\item [(ii)] We define
$W_k(\gamma|\delta)$ to be the $F[(N_k/P_k)\wr\sym{m}]$-module
obtained from
%\begin{align*}
\[ \begin{split}
\Res^{\sym{k}\wr\sym{m}}_{N_k\wr\sym{m}}{}&{} V_k(\gamma|\delta) \\
 {}&\cong{} \Ind_{N_k\wr(\sym{m_1}\times \sym{m_2})}^{N_k\wr\sym{m}}\left( \Inf^{N_k\wr\sym{m_1}}_{\sym{m_1}}(M^\gamma)\boxtimes \bigl( \Inf^{N_k\wr\sym{m_2}}_{\sym{m_2}}(M^\delta)\otimes \widehat{\sgn(N_k)}^{\otimes m_2}\big )
 \right)
 \end{split}
 \]
via the canonical surjection $N_k \wr \sym{m} \rightarrow (N_k\wr\sym{m})/(P_k)^m\cong (N_k/P_k)\wr\sym{m}$.

\item [(iii)] For $k\geq 2$,
we define $\overline{W}_k(\gamma|\delta)$ to be the $F[C_2\wr\sym{m}]$-module obtained
from $W_k(\gamma|\delta)$  via the canonical surjection
$(N_k/P_k) \wr \sym{m} \rightarrow ((N_k/P_k)\wr \sym{m})/(N_{\alt{k}}(P_k)/P_k)^m\cong C_2\wr \sym{m}$. We define
\[\overline{W}_{\!1}(\gamma|\delta)=\Inf^{C_2\wr\sym{m}}_{\sym{m}}W_1(\gamma|\delta).\]
\end{enumerate}
\end{defn}

We note that $W_k(\gamma | \delta)$ may equivalently be defined to be
the  $F[(N_k/P_k)\wr\sym{m}]$-module obtained from
\begin{equation}
\label{eq:W} %W_k(\gamma | \delta)
%=
\Ind_{N_k\wr(\sym{\gamma}\times\sym{\delta})}^{N_k\wr\sym{m}}\bigl(F(N_k\wr \sym{\gamma})\boxtimes\widehat{\sgn(N_k)}^{\otimes\delta}\bigr)
\end{equation}
via the canonical surjection as in Definition \ref{D:VW}(ii). %$(N_k\wr\sym{m})/(P_k)^m\cong (N_k/P_k)\wr\sym{m}$.
We have
 $W_1(\gamma|\delta)=V_1(\gamma|\delta) \cong M(\gamma|\delta)$ as $F\sym{m}$-modules.
When $k\geq 2$, the $F[C_2\wr\sym{m}]$-module $\overline{W}_k(\alpha|\beta)$ is
isomorphic to $V_k(\alpha|\beta)$, considered as a $F[C_2 \wr \sym{m}]$-module via the canonical surjection
$\sym{k} \wr \sym{m} \rightarrow
(\sym{k}\wr\sym{m})/\alt{k}^{m}\cong (\sym{k}/\alt{k})\wr\sym{m}\cong C_2\wr\sym{m}$.
%where $\alt{m}$ denote the alternating group of degree~$m$.
Similarly, $\overline{W}_k(\alpha|\beta)$ is
isomorphic to
$\Res^{\sym{k}\wr\sym{m}}_{N_k\wr\sym{m}}V_k(\alpha|\beta)$,
considered as a $F[C_2 \wr \sym{m}]$-module via the canonical surjection
$N_k \wr \sym{m} \rightarrow
(N_k\wr \sym{m})/(N_{\alt{k}}(P_k))^m\cong C_2\wr \sym{m}$.

\begin{lem}\label{L:barWk} For all $k\geq 2$ and all $(\gamma |\delta)\in \C^2(m)$, we have
\[\overline{W}_k(\gamma|\delta)\cong V_2(\gamma|\delta),\]
as $F[C_2\wr \sym{m}]$-modules.
\end{lem}

\begin{proof} It suffices to show that $\widehat{\sgn(k)}^{\otimes m_2}\cong \widehat{\sgn(2)}^{\otimes m_2}$ as $F[C_2\wr\sym{m_2}]$-modules, where $\sgn(k)$ is regarded as 
an $FC_2$-module via the canonical surjection
$\sym{k} \rightarrow \sym{k}/\alt{k}\cong C_2$. This is clear since  $\sgn(k)\cong \sgn(2)$ as
$FC_2$-modules in this regard.
\end{proof}

The following notation will be used to describe the direct summands of the Brauer quotients
of the signed Young permutation modules $M(\alpha|\beta)$.

\begin{notation}\label{N:Omega} Let $(\alpha|\beta)\in\C^2(n)$ and let $\rho=(1^{n_0},p^{n_1},(p^2)^{n_2},\ldots, (p^r)^{n_r}) \in \C(n)$.
We write $\LambdaABR$ for the set consisting of all pairs
 of tuples
 of compositions $(\mbf{\gamma}|\mbf{\delta})=(\mbf{\gamma}_0,\mbf{\gamma}_1,\ldots,\mbf{\gamma}_r|\mbf{\delta}_0,\mbf{\delta}_1,\ldots,\mbf{\delta}_r)$ such that:

\smallskip

\begin{enumerate}
\item $\alpha=\sum_{i=0}^rp^i\pdot\mbf{\gamma}_i$, $\beta=\sum_{i=0}^rp^i\pdot \mbf{\delta}_i$, and

\smallskip
\item $|\mbf{\gamma}_i|+|\mbf{\delta}_i|=n_i$ for each $i\in\{0,\ldots,r\}$.
\end{enumerate}
\end{notation}

Let $(\alpha|\beta)\in\C^2(n)$. Recall that $\Omega(\alpha|\beta)$ is the basis of $M(\alpha|\beta)$ consisting of all $(\alpha|\beta)$-tabloids.
As remarked after Corollary~\ref{C:basisBroueM}, an $F$-basis of $M(\alpha|\beta)(P_\rho)$ is obtained by
taking those $(\alpha|\beta)$-tabloids $\{(\T_+|\T_-)\}\in\Omega(\alpha|\beta)$
such that the rows of $\T_+$ and $\T_-$ are unions of the orbits of~$P_\rho$. %These are in fact precisely those $(\alpha|\beta)$-tabloids that are fixed under the action of $P_\rho$ (as prescribed by Corollary \ref{cor:Brauer}).
Given such a basis element $\{(\T_+|\T_-)\}$ and $i\in\{0,\ldots,r\}$, let $(\mbf{\gamma}_i)_j$
and~$(\mbf{\delta}_i)_k$ be the numbers of $P_\rho$-orbits of
length $p^i$ in rows $j$ and $k$ of $\T_+$ and $\T_-$, respectively.
For each $i\in\{0,\ldots,r\}$, let
\begin{align*}
\mbf{\gamma}_i&=\bigl( (\mbf{\gamma}_i)_1,(\mbf{\gamma}_i)_2,\ldots \bigr),\\
\mbf{\delta}_i&=\bigl( (\mbf{\delta}_i)_1,(\mbf{\delta}_i)_2,\ldots \bigr).
\end{align*}
Note that $|\mbf{\gamma}_i|+|\mbf{\delta}_i|=n_i$ for each $i$, and so
$(\mbf{\gamma}_0,\ldots,\mbf{\gamma}_r|\mbf{\delta}_0,\ldots,\mbf{\delta}_r)\in\LambdaABR$.
We say that the $(\alpha|\beta)$-tabloid $\{(\T_+|\T_-)\}$ is of {\sl $\rho$-type $(\mbf{\gamma}|\mbf{\delta})$}. We denote the set of all $(\alpha|\beta)$-tabloids of $\rho$-type $(\mbf{\gamma}|\mbf{\delta})$ by
$\Omega\bigl((\alpha|\beta),\rho\bigr)_{(\mbf{\gamma}|\mbf{\delta})}$. Then the disjoint union
\begin{equation}\label{eqn union}
\Omega\bigl((\alpha|\beta),\rho\bigr)=\bigcup_{(\mbf{\gamma}|\mbf{\delta})\in \Lambda((\alpha|\beta),\rho)}\Omega\bigl((\alpha|\beta),\rho)_{(\mbf{\gamma}|\mbf{\delta})}
\end{equation}
is an $F$-basis of $M(\alpha|\beta)(P_\rho)$. Thus, as $F$-vector spaces, we have
\begin{equation}\label{eqn M dec}
M(\alpha|\beta)(P_\rho)=F\Omega\bigl( (\alpha|\beta),\rho\bigr)=\bigoplus_{(\mbf{\gamma}|\mbf{\delta})\in \Lambda((\alpha|\beta),\rho)}F\Omega\bigl( (\alpha|\beta),\rho\bigr)_{(\mbf{\gamma}|\mbf{\delta})}\,.
\end{equation}
It is clear that (\ref{eqn M dec}) is in fact a decomposition of $FN_\rho$-modules, since $N_\rho$ permutes orbits of $P_\rho$ of the same size as blocks for its action, and therefore preserves the $\rho$-type in its action on $(\alpha|\beta)$-tabloids. Furthermore, $P_\rho$ fixes all $(\alpha|\beta)$-tabloids having a specified
$\rho$-type. Therefore we obtain the following lemma.

\begin{lem}\label{L:rhotype}
Let $(\alpha|\beta)\in\C^2(n)$ and $\rho=(1^{n_0},p^{n_1},\ldots,(p^r)^{n_r})$ be
a partition of $n$. The Brauer quotient of $M(\alpha|\beta)$ with respect to the subgroup $P_\rho$ has the following direct sum decomposition into $F[N_\rho/P_\rho]$-modules:
\[M(\alpha|\beta)(P_\rho)=\bigoplus_{(\mbf{\gamma}|\mbf{\delta})\in \Lambda((\alpha|\beta),\rho)}F\Omega((\alpha|\beta),\rho)_{(\mbf{\gamma}|\mbf{\delta})}.\]
\end{lem}

By Lemma \ref{L:rhotype}, to understand
the Brauer quotient $M(\alpha|\beta)(P_\rho)$ of the signed Young permutation module $M(\alpha|\beta)$,
it suffices to understand each of the $F[N_\rho/P_\rho]$-modules
$F\Omega\bigl( (\alpha|\beta),\rho \bigr)_{(\mbf{\gamma}|\mbf{\delta})}$.

%Recall the notation $\O_{ij}$ from the end of Section \ref{Sec:Syl}.

\begin{defn}\label{D:theta}
Suppose that $(\alpha|\beta)\in\C^2(n)$ and that $\rho=(1^{n_0},p^{n_1},\ldots,(p^r)^{n_r})$
is a partition of $n$.
Let the orbits of $P_\rho$ of size $p^i$ be $\mathcal{O}_{i,1},\ldots,\mathcal{O}_{i,n_i}$.
Let
\[\Theta:\Omega\bigl( (\alpha|\beta),\rho\bigr)\to \bigcup_{(\mbf{\gamma}|\mbf{\delta})\in \Lambda((\alpha|\beta),\rho)}\, \prod_{i=0}^r\Omega(\mbf{\gamma}_i|\mbf{\delta}_i)\] be the bijective function defined as follows. Suppose that $\{\T\}\in\Omega(\alpha|\beta)$ is of $\rho$-type $(\mbf{\gamma}|\mbf{\delta})$. For each $0\leq i\leq r$, let $\{\T_i\}$ be the $(\mbf{\gamma}_i|\mbf{\delta}_i)$-tabloid such that $\T_i$ is row standard, and row $k$ of $(\T_i)_+$ (respectively, $(\T_i)_-$) contains $j$ if and only if row $k$ of $\T_+$ (respectively, $\T_-$) contains the orbit $\O_{i,j}$. Define $\Theta(\{\T\})=(\{\T_i\})_{i=0,1,\ldots,r}$.
\end{defn}

We note that, by definition of $P_\rho$,
$$\mathcal{O}_{i,j}=\Bigl\{ (j-1)p^i+1+\sum_{\ell=0}^{i-1}n_\ell p^\ell,\ldots,jp^i+\sum_{\ell=0}^{i-1}n_\ell p^\ell \Bigr\},$$
for $i\in\{0,\ldots,r\}$ and $j\in\{1,\ldots,n_i\}$.
Clearly, the bijection $\Theta$ in Definition \ref{D:theta} restricts to a bijection,
also denoted $\Theta$,
 \[\Theta:\Omega\bigl( (\alpha|\beta),\rho\bigr)_{(\mbf{\gamma}|\mbf{\delta})}\to \prod_{i=0}^r\Omega(\mbf{\gamma}_i|\mbf{\delta}_i).\] Since $|\Omega(\mbf{\gamma}_i|\mbf{\delta}_i)|=\dim_F M(\mbf{\gamma}_i|\mbf{\delta}_i)=[\sym{n_i}:(\sym{\mbf{\gamma}_i}\times \sym{\mbf{\delta}_i})]$, we obtain the following lemma.

\begin{lem}\label{lemma size O}
Let $(\alpha|\beta)\in\C^2(n)$, $\rho=(1^{n_0},p^{n_1},\ldots,(p^r)^{n_r})$ such that $|\rho|=n$, and let $(\mbf{\gamma}|\mbf{\delta})\in \LambdaABR$. Set
$$H=\prod_{i=0}^r N_{p^i}\wr(\sym{\mbf{\gamma}_i}\times \sym{\mbf{\delta}_i})=\prod_{i=0}^r(N_{p^i}\wr\sym{\mbf{\gamma}_i})\times (N_{p^i}\wr\sym{\mbf{\delta}_i})\leq N_\rho\,.$$
Then $|\Omega((\alpha|\beta),\rho)_{(\mbf{\gamma}|\mbf{\delta})}|=[N_\rho:H]$.
\end{lem}

We have reached the main result of this section.

\begin{prop}\label{P:Broueperm} Suppose that $(\alpha|\beta)\in\C^2(n)$ and
that $\rho=(1^{n_0},p^{n_1},\ldots,(p^r)^{n_r}) \in \C(n)$. % is a partition of $n$.
Regarded as an $F[N_\rho/P_\rho]$-module, the Brauer quotient  $M(\alpha|\beta)(P_\rho)$
of the signed Young permutation module $M(\alpha|\beta)$
with respect to $P_\rho$ satisfies
\[M(\alpha|\beta)(P_\rho)  \cong\!\!\bigoplus_{(\mbf{\gamma}|\mbf{\delta})\in\Lambda((\alpha|\beta),\rho)} W_1(\mbf{\gamma}_0|\mbf{\delta}_0)\boxtimes W_p(\mbf{\gamma}_1|\mbf{\delta}_1)\boxtimes \cdots\boxtimes W_{p^r}(\mbf{\gamma}_r|\mbf{\delta}_r).\] %where the sum runs over tuples of compositions $(\mbf{\gamma}_0,\mbf{\gamma}_1,\ldots|\mbf{\delta}_0,\mbf{\delta}_1,\ldots)$ such that $\alpha=\sum_{i\geq 0}p^i\cdot \mbf{\gamma}_i$, $\beta=\sum_{i\geq 0}p^i\cdot \mbf{\delta}_i$ and, for all $i\geq 0$, we have $|\mbf{\gamma}_i|+|\mbf{\delta}_i|=n_i$.
\end{prop}
\begin{proof} Recall that for each $(\lambda|\mu) \in \C^2(n)$,
we have defined a row-standard $(\lambda|\mu)$-tableau $\T^{\lambda|\mu}$  immediately after
Definition~\ref{defn:pairTableau}.
Fix $(\mbf{\gamma}|\mbf{\delta})\in\Lambda((\alpha|\beta), \rho)$ and let $Z=F\Omega\bigl((\alpha|\beta),\rho\bigr)_{(\mbf{\gamma}|\mbf{\delta})}$. By Lemma \ref{L:rhotype}, it suffices to show that
$Z\cong \largeboxtimes_{i=0}^r W_{p^i}(\mbf{\gamma}_i|\mbf{\delta}_i)$ as $FN_\rho$-modules with
 $P_\rho$ acting trivially, or equivalently, by~\eqref{eq:W}, that
 \begin{equation}\label{eq:Z}
 Z\cong \largeboxtimes_{i=0}^r \Ind_{N_{p^i}\wr(\sym{\mbf{\gamma}_i}\times\sym{\mbf{\delta}_i})}^{N_{p^i}\wr\sym{n_i}}\big(F(N_{p^i}\wr \sym{\mbf{\gamma}_i})\boxtimes\widehat{\sgn(N_{p^i})}^{\otimes\mbf{\delta}_i}\big). \end{equation}
 Let $\{\mathrm{S}\}\in\Omega\bigl((\alpha|\beta),\rho\bigr)_{(\mbf{\gamma}|\mbf{\delta})}$ be the unique $(\alpha|\beta)$-tabloid such that \[\Theta(\{\mathrm{S}\})=(\T^{\mbf{\gamma}_0|\mbf{\delta}_0},\T^{\mbf{\gamma}_1|\mbf{\delta}_1},\ldots,\T^{\mbf{\gamma}_r|\mbf{\delta}_r})\in \prod_{i=0}^r\Omega(\mbf{\gamma}_i|\mbf{\delta}_i).\]
 Using the $N_\rho$-action on $Z$, we observe that $Z$ is a cyclic $FN_\rho$-module generated by $\{\mathrm{S}\}$. Let~$X$ be the subspace of $Z$ linearly spanned by
  $\{\mathrm{S}\}$. By the definition of $\{\mathrm{S}\}$, the subspace~$X$ is an $FH$-module where
  \[H=\prod_{i=0}^r N_{p^i}\wr(\sym{\mbf{\gamma}_i}\times\sym{\mbf{\delta}_i})=\prod_{i=0}^r\bigl( (N_{p^i}\wr\sym{\mbf{\gamma}_i})\times (N_{p^i}\wr\sym{\mbf{\delta}_i}) \bigr) \le N_\rho,\]
  and there is an isomorphism
  \[X\cong \big(F(N_{1}\wr\sym{\mbf{\gamma}_0})\boxtimes  \widehat{\sgn(N_{1})}^{\otimes \mbf{\delta}_0}\big)\boxtimes\cdots \boxtimes \big(F(N_{p^r}\wr\sym{\mbf{\gamma}_r})\boxtimes \widehat{\sgn(N_{p^r})}^{\otimes \mbf{\delta}_r}\big)\]
  of $FH$-modules. Since $\dim_FZ=[N_\rho:H]\dim_F X$ by Lemma \ref{lemma size O}, we have $Z\cong \Ind^{N_\rho}_HX$ by
  the characterization of induced modules in
\cite[Section 8, Corollary 3]{Alperin}.
Hence we obtain the isomorphism~\eqref{eq:Z} as desired.
\end{proof}

%We remark that the special case of Proposition~\ref{P:Broueperm} when $\beta = \varnothing$
%is proved as Proposition~1 in \cite{ErdSchr}.
%\cite[Proposition 1]{ErdSchr}.

%\begin{cor}\label{cor:YoungBroue}
%Let $\alpha$ be a composition of $n$ and let $\rho = (1^{n_0},p^{n_1},\ldots,(p^r)^{n_r})$
%be a partition of $n$. Regarded as a modules for
%$\sym{n_0} \times \sym{n_1} \times \cdots \times \sym{n_r}$ there is an isomorphism
%\[ M^{\alpha}(P_\rho) \cong \bigoplus_{(\mbf{\gamma}|\varnothing) \in \Lambda((\alpha|\varnothing),
%\rho)} M^{\mbf{\gamma}_0} \boxtimes M^{\mbf{\gamma}_1} \boxtimes \cdots \boxtimes M^{\mbf{\gamma}_r}. \]
%\end{cor}
%
%\begin{proof}
%Note that $(\mbf{\gamma}|\mbf{\delta}) \in \Lambda((\alpha | \varnothing), \rho)$ only if
%$\mbf{\delta} = (\varnothing, \ldots, \varnothing)$. Therefore Proposition~\ref{P:Broueperm}
%implies that
%\[ M^{\alpha}(P_\rho) \cong
%\bigoplus_{(\mbf{\gamma}|\varnothing) \in \Lambda((\alpha|\varnothing),
%\rho)} W_1(\mbf{\gamma}_0|\varnothing)\boxtimes W_p(\mbf{\gamma}_1|\varnothing)
%\boxtimes \cdots\boxtimes W_{p^r}(\mbf{\gamma}_r|\varnothing).\]
%As seen in Corollary~\ref{C:basisBroueM}(ii),
%the base group in the factor $N_{p^i}/P_{p^i} \wr \sym{n_i}$ of the product defining $N_\rho$ acts
%trivially on $W_i(\mbf{\gamma_i}|\varnothing) = \Inf_{S_{n_i}}^{N_k \wr S_{n_i}} (M^{\mbf{\gamma}_i})$.
%Therefore it is admissible to regard
%the left-hand side as a module for $\sym{n_0} \times \sym{n_1} \times \cdots \times \sym{n_r}$
%and the corollary follows.
%\end{proof}

\section{Young modules and signed Young modules}
\label{S:YoungConstruction}

In this section we define Young modules and signed Young modules in the setting of the symmetric group
and prove Theorem~\ref{T:Donkin}.

\subsection{Vertices}
As a first step we identify the possible vertices of summands of signed Young modules.
Recall from Subsection~\ref{Sec:Syl} that $P_k$ denotes a Sylow subgroup
of~$\sym{k}$ and, if $\rho$ is a partition of $n$, then $P_\rho$
denotes a Sylow subgroup of the Young subgroup~$\sym{\rho}$ of~$\sym{n}$.
We require the following lemma from \cite{ErdmannYoung}; a proof, slightly shorter than the one in \cite{ErdmannYoung},
is included to make the article self-contained.

\begin{lem}[Erdmann \protect{\cite[Lemma 1]{ErdmannYoung}}]
\label{lem:Erdmann}
Let $G$ be a finite group and let $M$ be a $p$-permutation $FG$-module.
If $P$ and $\widetilde{P}$ are $p$-subgroups of $G$
such that $P < \widetilde{P}$ and $\dim M(P) = \dim M(\widetilde{P})$
then no indecomposable summand of $M$ has vertex $P$.
\end{lem}

\begin{proof}
Suppose, for a contradiction, that $U$ is such a summand. Let $M = U \oplus V$
where $V$ is a complementary $FG$-module. By Corollary~\ref{cor:Brauer}, we have $U(P) \not=0$
and $U(\widetilde{P}) = 0$. Thus
\[  M(\widetilde{P}) = U(\widetilde{P}) \oplus V(\widetilde{P}) = V(\widetilde{P}),\ \ \text{and}\ \ M(P) = U(P) \oplus V(P). \]
This is a contradiction, since taking a $p$-permutation basis for $V$ and applying
Corollary~\ref{cor:Brauer} shows that $\dim V(P) \ge \dim V(\widetilde{P})$.
\end{proof}

\begin{prop}
\label{P:vertices}
Let $(\alpha | \beta) \in \C^2(n)$. If $P$ is a vertex of an indecomposable summand of $M(\alpha|\beta)$
then there exists $\rho = (1^{n_0},p^{n_1},\ldots, (p^r)^{n_r}) \in \C(n)$ such that
$P$ is conjugate in $\sym{n}$ to $P_\rho$.
\end{prop}

\begin{proof}
Let $H$ be the Young subgroup of $\sym{n}$ having the same orbits as $P$ on $\{1,\ldots, n\}$
and let $\widetilde{P}$ be a Sylow $p$-subgroup of $H$. Note that $\widetilde{P}$ has the same orbits
on $\{1,\ldots,n\}$ as~$P$: suppose that each subgroup has exactly $n_i$ orbits of size~$p^i$ for
each $i \in \{0,\ldots,r\}$, so $\widetilde{P}$
is conjugate in $\sym{n}$ to $P_\rho$. It suffices to prove that $P = \widetilde{P}$.

Let $\{\T\}$ be an $(\alpha|\beta)$-tabloid fixed by $P$. As remarked following Corollary~\ref{C:basisBroueM},
each row of $\T$ is a union of orbits of $P$. Therefore each row is a union of orbits of $\widetilde{P}$,
and so if $g \in \widetilde{P}$ then $g \{\T\} = \pm \{\T\}$. Since $g$ has $p$-power order,
we see that $g  \{\T\} = \{\T\}$. It now follows from
Corollary~\ref{C:basisBroueM} that $\dim M(\alpha|\beta)(P) =\dim M(\alpha|\beta)(\widetilde{P})$.
By Lemma~\ref{lem:Erdmann} we have $P = \widetilde{P}$, as required.
\end{proof}

Combining Proposition~\ref{P:Broueperm} and Proposition~\ref{P:vertices}, we see that
the Brou{\'e} correspondents of the non-projective indecomposable summands of
$M(\alpha|\beta)$ are certain outer tensor products of the projective
indecomposable summands of the $F[(N_{p^i}/P_{p^i} )\wr \sym{m}]$-modules $W_{p^i}(\gamma|\delta)$
in Definition~\ref{D:VW}.
 In fact, it is most convenient to factor out a further subgroup that acts trivially, and
consider projective summands of the $F[C_2 \wr \sym{m}]$-modules
$\overline{W}_{\!p^i}(\gamma|\delta)$.

%Version 3 planned to do Young modules here (move above paragraph into projective summand section)

\subsection{Projective summands of $\overline{W}_{k}(\gamma|\delta)$}\label{Sec:Qbar}

Fix $k \in \N$ and $m_1, m_2 \in \N_0$. Let $m = m_1+m_2$.
Recall from Subsection~\ref{Sec:Modules}
that if $\alpha \in \RP(n)$, that is,~$\alpha$ is a $p$-restricted partition of $n$,
then  $\Palpha$ denotes the projective cover of the simple $F\sym{n}$-module~$D_\alpha$.

Let $\mathcal{F}$ denote the bifunctor sending a pair $(U|V)$ where $U$ is an $F\sym{m_1}$-module
and $V$ is an $F\sym{m_2}$-module  to
\[ \Ind_{C_2\wr(\sym{m_1}\times\sym{m_2})}^{C_2\wr\sym{m}}\left
(\Inf^{C_2\wr\sym{m_1}}_{\sym{m_1}}(U)\boxtimes \bigl(\Inf^{C_2\wr\sym{m_2}}_{\sym{m_2}}(V)\otimes \widehat{\sgn(2)}^{\otimes m_2}\bigr)\right ).\]
For example, the module $\overline{W}_k(\gamma|\delta)$ in Definition~\ref{D:VW} satisfies
$\overline{W}_k(\gamma|\delta) \cong \mathcal{F}(M^\gamma|M^\delta)$.

\begin{defn}\label{def:Qbar}
Let $(\alpha|\beta) \in \RP^2(m)$.
We define $\overline{Q}(\alpha|\beta) = \mathcal{F}(\Palpha|\Pbeta)$.
\end{defn}

Thus, by definition
\[ \overline{Q}(\alpha|\beta)\cong \Ind_{C_2\wr(\sym{m_1}\times \sym{m_2})}^{C_2\wr\sym{m}}\left(
\Inf^{C_2\wr\sym{m_1}}_{\sym{m_1}}(\Palpha)\boxtimes \bigl(\Inf^{C_2\wr\sym{m_2}}_{\sym{m_2}}(\Pbeta)\otimes \widehat{\sgn(2)}^{\otimes m_2}\bigr) \right). \]
Example~\ref{eg:projectives} gives an example of these modules.
Note that each tensor factor is projective, so each $\overline{Q}(\alpha|\beta)$ is projective.

%When $k=1$ we identify $C_2 \wr \sym{m}$ with $\sym{m}$ and get
%\[Q_1(\alpha|\beta)=\Ind^{\sym{m}}_{\sym{m_1}\times\sym{m_2}}(\Palpha\boxtimes (\Pbeta\otimes\sgn(m_2))). \] %=R_1(\alpha|\beta). \]

\begin{lem}
\label{lem:Qcover}
The $F[C_2 \wr \sym{m}]$-modules
\[ \mathcal{F}(D_\alpha|D_\beta)
= \Ind_{C_2\wr(\sym{m_1}\times\sym{m_2})}^{C_2\wr\sym{m}}\left
(\Inf^{C_2\wr\sym{m_1}}_{\sym{m_1}}(D_\alpha)\boxtimes \bigl(\Inf^{C_2\wr\sym{m_2}}_{\sym{m_2}}(D_\beta)\otimes \widehat{\sgn(2)}^{\otimes m_2}\bigr)\right ) \]
for $(\alpha|\beta) \in \RP^2(m)$ form a complete
set of non-isomorphic simple $F[C_2 \wr \sym{m}]$-modules. Moreover,
the $F[C_2 \wr \sym{m}]$-module $\overline{Q}(\alpha|\beta)$ is the  projective cover of
$\mathcal{F}(D_\alpha|D_\beta)$ and the modules $\overline{Q}(\alpha|\beta)$ for
$(\alpha|\beta) \in \RP^2(m)$ form a complete set of non-isomorphic indecomposable
projective modules for $F[C_2 \wr \sym{m}]$.
\end{lem}

\begin{proof}
The first claim follows from the construction of simple modules for wreath products
stated in Theorem 4.34 of \cite{JK}. For the second, note
that by functoriality, there is a surjection $\overline{Q}(\alpha|\beta) = \mathcal{F}(\Palpha|\Pbeta)
\rightarrow \mathcal{F}(D_\alpha|D_\beta)$. Hence the projective $F[C_2 \wr \sym{m}]$-module
$\overline{Q}(\alpha|\beta)$ has the projective cover of $\mathcal{F}(D_\alpha|D_\beta)$
as a summand. Since the inertial group of
\[ \Inf^{C_2\wr\sym{m_1}}_{\sym{m_1}}(P^\alpha)\boxtimes \bigl(\Inf^{C_2\wr\sym{m_2}}_{\sym{m_2}}(P^\beta)\otimes \widehat{\sgn(2)}^{\otimes m_2}\bigr) \]
is $C_2 \wr \sym{m_1} \times C_2 \wr \sym{m_2}$,
it follows from~\cite[Proposition 3.13.2]{Benson} that $\overline{Q}(\alpha|\beta)$ is indecomposable.
Therefore $\overline{Q}(\alpha|\beta)$ is the projective cover of $\mathcal{F}(D_\alpha|D_\beta)$.
\end{proof}

Let $G$ be a finite group. By Section 3.11 in \cite{Benson}, we may associate a character to a $p$-permutation
$FG$-module $M$
by taking a $p$-modular system $(K,\mathcal{O},F)$ compatible with $F$
and an $\mathcal{O}G$-module $M_\mathcal{O}$ whose $p$-modular reduction is $M$.
The \emph{ordinary character} of $M$ is then the character of the $KG$-module $K \otimes_\mathcal{O}
M_\mathcal{O}$.
If $M$ is projective and indecomposable, the ordinary character of $M$ may
equivalently be defined by Brauer reciprocity (see for instance \cite[Section 15.4]{Serre}).

\begin{prop}\label{P:dominates}
\label{P:Wsummands}
%Let $k \in \N$ and let $(\alpha|\beta) \in \RP^2(m)$ where $m \in \N_0$.
%Let $m_1 = |\alpha|$ and let $m_2 = |\beta|$.
Let $(\gamma|\delta) \in \P^2(m)$ where $|\gamma| = m_1$ and $|\delta| = m_2$.
Each indecomposable projective summand of $\overline{W}_k(\gamma|\delta)$
is isomorphic to some $\overline{Q}(\alpha|\beta)$, where $(\alpha |\beta) \in \RP^2(m)$ satisfies
\begin{itemize}
\item[(i)] $|\alpha| = m_1$ and $|\beta| = m_2$;
\item[(ii)] $\alpha \unrhd \gamma$ and $\beta \unrhd \delta$.
\end{itemize}
%Moreover $[\overline{W}_k(\gamma|\delta) | \overline{Q}_k(\gamma|\delta)] = 1$.
%{\bf [This is not clear: maybe can get by induction when we need it.]}
\end{prop}

\begin{proof}

By Lemma~\ref{lem:Qcover}, each indecomposable projective summand of
$\overline{W}_k(\gamma|\delta)$ is isomorphic to some $\overline{Q}(\alpha|\beta)$.
By the `wedge' shape of the decomposition matrix of $\sym{n}$
with columns labelled by $p$-restricted partitions (see for instance \cite[Theorem~5.2]{KleshchevAMS})
and Brauer reciprocity,
the ordinary character of $\Palpha$ contains the irreducible character $\chi^\alpha$
exactly once. %, and all other constituents are of the form $\chi^\sigma$ with $\sigma \unrhd \alpha$.
Hence the ordinary character of $\overline{Q}(\alpha|\beta)$
contains the character
\[ \chi^{(\alpha|\beta)} = \Ind_{C_2\wr(\sym{m_1}\times\sym{m_2})}^{C_2\wr\sym{m}}\left
(\Inf^{C_2\wr\sym{m_1}}_{\sym{m_1}} (\chi^\alpha) \times
(\Inf^{C_2\wr\sym{m_2}}_{\sym{m_2}}(\chi^\beta)\otimes \widehat{\sgn(2)}^{\otimes m_2})\right) \]
defined in Subsection~\ref{Sec:Modules}
exactly once.

We now consider when the ordinary character of $\overline{W}_k(\gamma|\delta)$  contains
$\chi^{(\alpha|\beta)}$.
The restriction of $\overline{W}_k(\gamma|\delta)$ to the base group in the wreath
product $C_2 \wr \sym{m}$ is a direct sum of $1$-dimensional submodules. In each such
submodule, $m_1$ of the factors in the product $C_2^m$ act trivially and $m_2$ of the factors act as $\sgn(2)$.
It follows by basic Clifford theory that
the ordinary character of $\overline{W}_k(\gamma|\delta)$
contains the character $\chi^{(\alpha|\beta)}$ only
if $|\alpha| = m_1$ and $|\beta| = m_2$.
By Young's rule (see for instance \cite[Theorem 13.13]{james1978representation}), the ordinary character
of $M^\gamma$ contains $\chi^\alpha$ only if $\alpha \unrhd \gamma$, and similarly
the ordinary character of $M^\delta$ contains $\chi^\beta$ only if $\beta \unrhd \delta$.
% moreover, $\chi^\gamma$
%appears exactly once.
%and all its other irreducible
%constituents are of the form $\chi^{(\sigma|\tau)}$ where $\sigma \unrhd \alpha$ and $\tau \unrhd \beta$.
It follows that if $\overline{Q}(\alpha|\beta)$ is a summand of $\overline{W}_k(\gamma|\delta)$ then
$\alpha \in \P(m_1)$, $\beta \in \P(m_2)$,
 $\alpha \unrhd \gamma$ and $\beta \unrhd \delta$.
 %Moreover, since $\chi^{(\gamma|\delta)}$
 %does not appear in the ordinary character of $\overline{Q}_k(\sigma|\tau)$ if
%either $\sigma \rhd \gamma$ or $\tau \rhd \delta$, we see that
%$[\overline{W}_k(\gamma|\delta) : \overline{Q}_k(\gamma|\delta)] = 1$. {[\bf No: this is assuming
%it must appear as a projective summand]}.
 \end{proof}

\subsection{Definition of signed Young modules}
We define signed Young modules as
the Brou{\'e} correspondents of
tensor products of suitable inflations of
the modules $\overline{Q}(\alpha|\beta)$. To make this precise, we need
the three further families of modules defined below: their definition follows
the same pattern as the $p$-permutation modules $V_k(\gamma|\delta)$, $W_k(\gamma|\delta)$
and $\overline{W}_k(\gamma|\delta)$ in Definition~\ref{D:VW}.

\begin{defn}\label{def:Rk}
Let $k \in \N$, let $m \in \N_0$, and let $(\alpha|\beta) \in \RP^2(m)$.
Let $m_1 = |\alpha|$ and $m_2 = |\beta|$.
The $F[\sym{k}\wr\sym{m}]$-module $R_k(\alpha|\beta)$ is defined by
\[R_k(\alpha|\beta)=\Ind_{\sym{k}\wr(\sym{m_1}\times\sym{m_2})}^{\sym{k}\wr\sym{m}}\left (\Inf^{\sym{k}\wr\sym{m_1}}_{\sym{m_1}}(\Palpha)\boxtimes \bigl(\Inf^{\sym{k}\wr\sym{m_2}}_{\sym{m_2}}(\Pbeta)\otimes \widehat{\sgn(k)}^{\otimes m_2}\bigr)\right ).\]
\end{defn}

By convention,
$R_k(\varnothing|\beta)=\Inf^{\sym{k}\wr\sym{m_2}}_{\sym{m_2}}(\Pbeta)\, \otimes\, \widehat{\sgn(k)}^{\otimes m_2}\!\!$, and similarly for $R_k(\alpha|\varnothing)$.
Furthermore if $m=0$, then $R_k(\varnothing|\varnothing)$ is the trivial $F\sym{0}$-module.
If $k=1$ then we identify $\sym{k}\wr\sym{m}$ with $\sym{m}$ and get
 \[R_1(\alpha|\beta)=\Ind^{\sym{m}}_{\sym{m_1}\times\sym{m_2}}\bigl( \Palpha\boxtimes (\Pbeta\otimes\sgn(m_2))
 \bigr).\]

Recall from Section~\ref{Sec:Syl}
that $P_k$ is a fixed Sylow $p$-subgroup of $\sym{k}$ and that $N_k = N_{\sym{k}}(P_k)$.

\begin{defn}\label{def:Qk}
Let $k \in \N$, let $m \in \N_0$, and let $(\alpha|\beta) \in \RP^2(m)$.
Let $Q_k(\alpha|\beta)$
 be the $F[(N_k/P_k)\wr \sym{m}]$-module defined
by
\[ Q_k(\alpha|\beta) = \Res^{\sym{k}\wr\sym{m}}_{N_k\wr\sym{m}} R_k(\alpha|\beta) \]
considered as an $F[(N_k/P_k)\wr \sym{m}]$-module via
the canonical surjection $N_k \wr \sym{m} \rightarrow
(N_k \wr \sym{m}) / (P_k)^m \cong (N_k / P_k) \wr \sym{m}$.
\end{defn}

%We note that when $k\ge 2$ the
%$F[C_2 \wr \sym{m}]$-module
%$\overline{Q}_k(\alpha|\beta)$ can also be obtained directly from $R_k(\alpha|\beta)$
%via the canonical surjection
%$(\sym{k}\wr\sym{m})/\alt{k}^{m}\cong (\sym{k}/\alt{k})\wr\sym{m}\cong C_2\wr\sym{m}$.
%Provided $k \ge 2$,
%the modules $R_k(\alpha|\beta)$ and $Q_k(\alpha|\beta)$ are indecomposable by \cite[Proposition 5.1]{BKulshammer}. Thus $\overline{Q}_k(\alpha|\beta)$ is also indecomposable.
%
%

Again if $k=1$ we identify $N_1 \wr \sym{m}$ with $\sym{m}$ and
we have $Q_1(\alpha|\beta) = R_1(\alpha|\beta)$.
Since $(\alt{k})^m$ acts trivially on $R_k(\alpha|\beta)$ we see that
$(N_{\alt{k}}(P_k)/P_k)^m$ acts trivially on $Q_k(\alpha|\beta)$.
It is clear that
\begin{equation}\label{eq:Qk}
Q_k(\alpha|\beta)\cong \Ind_{N_k\wr(\sym{m_1}\times \sym{m_2})}^{N_k\wr\sym{m}}
\bigl(\Inf^{N_k\wr\sym{m_1}}_{\sym{m_1}}(\Palpha)\boxtimes (\Inf^{N_k\wr\sym{m_2}}_{\sym{m_2}}
(\Pbeta)\otimes \widehat{\sgn(N_k)}^{\otimes m_2})\bigr),
\end{equation}
again regarded as an $F[(N_k/ P_k) \wr \sym{m}]$-module by this canonical surjection.

%Moreover, $(\sym{k} \wr \sym{m}) / (\alt{k} \wr \sym{m}) \cong
%(N_k \wr \sym{m}) / (N_{\alt{k}} \wr \sym{m}) \cong C_2 \wr \sym{m}$.

\begin{defn}\label{def:Qbark}
Let $k \in \N$, let $m \in \N_0$, and let $(\alpha|\beta) \in \RP^2(m)$.
For $k\ge 2$, let $\overline{Q}_k(\alpha|\beta)$ be the $F[C_2\wr\sym{m}]$-module
obtained from $Q_k(\alpha|\beta)$ via the canonical surjection
\[ (N_k/P_k) \wr \sym{m} \rightarrow ((N_k/P_k)\wr \sym{m})/(N_{\alt{k}}(P_k)/P_k)^m\cong C_2\wr \sym{m}. \]
We define the  $F[C_2\wr\sym{m}]$-module $\overline{Q}_1(\alpha|\beta)$ by
\[ \overline{Q}_1(\alpha|\beta)=\Inf^{C_2\wr\sym{m}}_{\sym{m}}Q_1(\alpha|\beta).\]
\end{defn}

The following lemma justifies the notation $\overline{Q}_k(\alpha|\beta)$ for
the projective modules just defined.
%The first isomorphism is clear from the
%definitions, while the second may be shown as in Lemma~\ref{L:barWk}.
%We therefore leave the proof to the reader.

\begin{lem}\label{L:barQk} Let $k\geq 2$ and let $(\alpha|\beta)\in\RP^2(m)$ where $m\in\N_0$.
Then
 \[\overline{Q}_k(\alpha|\beta)\cong \overline{Q}(\alpha|\beta)
 \cong R_2(\alpha|\beta)\] as $F[C_2\wr\sym{m}]$-modules.
\end{lem}

\begin{proof}
The first isomorphism is clear from the
definitions and the second follows as in Lemma~\ref{L:barWk}.
\end{proof}

We pause to give a small example showing the
exceptional behaviour when $k=1$.

\begin{eg}\label{eg:projectives}
Let $p=3$, and let $k \ge 2$. Let $\varepsilon = \Inf_{\sym{3}}^{\sym{k} \wr \sym{3}} (\sgn(3))$.
There are four mutually non-isomorphic $1$-dimensional simple $F[\sym{k} \wr \sym{3}]$-modules, namely
\[ \widehat{F(k)}^{\otimes 3}, \
   \widehat{\sgn(k)}^{\otimes 3}, \
   \widehat{F(k)}^{\otimes 3} \!\otimes \varepsilon, \
   \widehat{\sgn(k)}^{\otimes 3} \!\otimes \varepsilon, \]
where the trivial module appears as $\widehat{F(k)}^{\otimes 3} \! \otimes \varepsilon$.
The projective covers of these modules are $R_k((1,1,1)|\varnothing)$, $R_k(\varnothing|(2,1))$,
$R_k((2,1) | \varnothing)$, $R_k( \varnothing | (1,1,1))$, respectively. Quotienting
out by the trivial action of $\alt{k}$, the
corresponding modules $\overline{Q}( \alpha | \beta)$ for $F[C_2 \wr \sym{3}]$
are precisely the projective covers of the four one-dimensional simple modules
for $F[C_2 \wr \sym{3}]$.
The four remaining simple modules for $F[C_2 \wr \sym{3}]$,
each projective; by Lemma~\ref{lem:Qcover}, they are
isomorphic to the modules $\overline{Q}(\alpha|\beta)$ where both $\alpha$ and $\beta$ are non-empty.
By contrast, when $k=1$, identifying $\sym{1} \wr \sym{3}$ with $\sym{3}$ as described after
Definition~\ref{def:Rk}, we have $\overline{Q}_1((1,1,1) | \varnothing) \cong
\overline{Q}_1(\varnothing |(2,1)) \cong P^{(1,1,1)} \cong M^{(2,1)} \otimes \sgn$ and
$\overline{Q}_1((2,1) | \varnothing) \cong \overline{Q}_1(\varnothing|(1,1,1)) \cong P^{(2,1)}
\cong M^{(2,1)}$.
\end{eg}

We are finally ready to define signed Young modules.

\begin{defn}\label{D:signedYoung}
Let $(\lambda|p\mu) \in \RP^2(n)$. Let $\lambda = \sum_{i\ge 0} p^i\lambda(i)$ and
$\mu = \sum_{i \ge 0} p^i\mu(i)$ be the $p$-adic expansions of $\lambda$ and $\mu$, as defined
in~\eqref{eq:padic}. Let $n_0 = |\lambda(0)|$ and let
$n_i = |\lambda(i)| + |\mu(i-1)|$ for each $i \in \N$. Let $r$ be maximal such that $n_r \not= 0$
and let $\rho = (1^{n_0},p^{n_1},\ldots, (p^r)^{n_r})$.
We define the \emph{signed Young module} $Y(\lambda|p\mu)$ to be
the unique (up to isomorphism) $F\sym{n}$-module $V$ such that
\[ V(P_\rho) \cong Q_1\bigl(\lambda(0)|\varnothing\bigr) \boxtimes Q_p\bigl(\lambda(1)|\mu(0)\bigr)
\boxtimes \cdots
\boxtimes Q_{p^r}\bigl(\lambda(r)|\mu(r-1)\bigr). \]
We define a \emph{Young module} to be a signed Young module
of the form $Y(\lambda|\varnothing)$.
\end{defn}

The isomorphism above is an isomorphism of projective
$F[\sym{n_0} \times (N_p/P_p) \wr \sym{n_1} \times \cdots \times (N_{p^r}/P_{p^r}) \wr \sym{n_r}]$-modules.
Observe that $P_\rho$ is trivial if and only if $\lambda$ is $p$-restricted and $\mu = \varnothing$;
in this case $Q_1(\lambda(0)|\varnothing)$ is regarded as a $\sym{n}$-module by identifying
$N_1 \wr \sym{n}$ with $\sym{n}$, and since $\lambda = \lambda(0)$
we have $Y(\lambda|\varnothing) = Q_1(\lambda|\varnothing) = \Plambda$.

The following proposition gives part of Theorem~\ref{T:Donkin}(i).

\begin{prop}{\ }
\label{P:summands}
\begin{thmlist}
\item If $(\alpha|\beta) \in \P^2(n)$ then $M(\alpha|\beta)$ is a direct sum of signed
Young modules~$Y(\lambda|p\mu)$.
\item If $\alpha \in \P(n)$ then $M^\alpha$ is a direct sum of Young
modules $Y(\lambda|\varnothing)$.
\end{thmlist}
\end{prop}

\begin{proof}
Let $(\alpha|\beta) \in \P^2(n)$ and let $V$ be an indecomposable summand of $M(\alpha|\beta)$.
By Proposition~\ref{P:vertices}
there exists $\rho = (1^{m_0},p^{m_1},\ldots, (p^r)^{m_r}) \in \C(n)$ such that $P_\rho$ is a
vertex of~$V$. %As in Corollary~\ref{C:basisBroueM}, let $N_\rho = N_{\sym{n}}(P_\rho)$;
Recall that
$N_\rho/P_\rho \cong \sym{n_0}\times ((N_p/P_p)\wr\sym{n_1})\times\cdots\times ((N_{p^r}/P_{p^r})\wr\sym{n_r})$.
By Proposition~\ref{P:Broueperm}, there
exists $(\mbf{\gamma}|\mbf{\delta}) \in \LambdaABR$
such that  the projective $F[N_\rho/P_\rho]$-module $V(P_\rho)$
is a direct summand of
\[ W_1(\mbf{\gamma}_0|\mbf{\delta}_0)\boxtimes W_p(\mbf{\gamma}_1|\mbf{\delta}_1)\boxtimes \cdots\boxtimes W_{p^r}(\mbf{\gamma}_r|\mbf{\delta}_r). \]
By Lemmas~\ref{lem:Qcover} and~\ref{L:barQk}
there exist partitions $\lambda(0), \ldots, \lambda(r)$ and $\mu(0), \ldots, \mu(r-1)$ such that
\[ V(P_\rho) = Q_1\bigl(\lambda(0)\bigr) \boxtimes Q_p\bigl(\lambda(1) | \mu(0)\bigr) \boxtimes
\cdots \boxtimes Q_{p^r}\bigl(\lambda(r) | \mu(r-1)\bigr). \]
By Theorem~\ref{BC1}, $V \cong Y(\lambda|p\mu)$ where $\lambda = \sum_{i=0}^r p^i \lambda(i)$ and
$\mu = \sum_{i=0}^{r-1} p^i \mu(i)$. This proves part (i).
For part (ii), observe that if $\beta = \varnothing$ then we have $\mbf{\delta}_i = \varnothing$ for each $i$,
and so $\mu(i) = \varnothing$ for each $i$.
\end{proof}

\subsection{Column symmetrization of $(\alpha|\beta)$-tabloids}

To deal with the projective summands of signed Young permutation modules we require the following corollary
of the key lemma used by James to prove his Submodule Theorem in \cite{james1978representation}.
Given a tableau $\t$ with entries from a set $\mathcal{O}$,
let $C_\t \le \sym{\mathcal{O}}$ be the group of permutations which fix the columns of $\t$
setwise. Set $\kappa_\t = \sum_{g \in C_\t} \sgn(g) g$.

\begin{prop}\label{P:Umodules}
Let $\lambda \in \P(n)$ and let $\t$ be a $\lambda$-tableau.
In any direct sum decomposition of $M^\lambda$ into
indecomposable modules there is a unique summand $U^\lambda$ such that $\kappa_\t U^\lambda \not= 0$.
Moreover if $\alpha \in \P(n)$ then $\kappa_\t U^\alpha = 0$ unless $\lambda \unrhd \alpha$.
\end{prop}

\begin{proof}
This follows immediately from \cite[Lemma 4.6]{james1978representation}.
\end{proof}

By the Krull--Schmidt Theorem, the $U^\lambda$  are well-defined  up to isomorphism. It is clear that
$U^\alpha \cong U^\beta$ if and only if $\alpha = \beta$.

We also need the following generalization of part of James' lemma.

\begin{lem}\label{lem:columnSym}
Let $(\alpha|\beta) \in \C^2(n)$ and let $\T = (\T_+|\T_-)$ be an $(\alpha|\beta)$-tableau.
Let $\lambda \in \P(n)$ and let $\t$ be a $\lambda$-tableau.
If $\kappa_t \{\T\} \not= 0$ then $(\lambda | \varnothing) \unrhd (\alpha|\beta)$.
\end{lem}

\begin{proof}
Let $m_1=|\alpha_1|$. Let $\mathcal{O}$ be the set of entries of $\T_+$. Let $H = \sym{\mathcal{O}} \cap \sym{\lambda'}$
and let $\mathcal{O}_1, \ldots, \mathcal{O}_s$ be the orbits of $H$ on $\mathcal{O}$,
ordered so that $|\mathcal{O}_1| \ge \ldots \ge |\mathcal{O}_s|$.
Let $\nu = (|\mathcal{O}_1|,\ldots,|\mathcal{O}_s|)' \in \P(m_1)$.
The $j$th largest orbit of $H$ has size at most $\lambda_j'$. Therefore $\nu_j'
\le \lambda_j'$ for each~$j \in \{1,\ldots,s\}$,
and so~$\nu$ is a subpartition of $\lambda$. It immediately follows that
\begin{equation}
\label{eq:lambdanu}
\sum_{i=1}^k \lambda_i \ge \sum_{i=1}^k \nu_i
\end{equation}
for all $k \in \N$. (By our standing convention, $\nu_i = 0$ if $i > \ell(\nu)$).

Let $\t^\star$ be a $\nu$-tableau having the entries of $\mathcal{O}_j$ in its $j$th column.
Observe that $C_{\t^\star} \le C_\t$. Choose $g_1,\ldots, g_s \in C_\t$ such
that $C_\t = g_1 C_{\t^\star}  \cup \ldots \cup g_sC_{\t^\star}$
where the union is disjoint. We have
\[ \kappa_{\t} = (\sgn(g_1) g_1 + \cdots + \sgn(g_s) g_s)\kappa_{\t^\star} . \]
Since $\kappa_\t \{\T\} \not= 0$ we have $\kappa_{\t^\star} \{\T\} \not= 0$.
Since $C_{\t^\star}$ fixes the entries in $\T_-$, it follows that
$\kappa_{\t^\star} \{\T_+\} \not= 0$. The argument used to prove Lemma 4.6 in \cite{james1978representation}
now shows that any two entries in the same row of $\{\T_+\}$ lie in different columns of $\t^\star$,
and so $\nu \unrhd \alpha$. Hence, by~\eqref{eq:lambdanu}, we have
$(\lambda |\varnothing) \unrhd (\alpha|\beta)$, as required.
\end{proof}

\subsection{Proof of Theorem~\ref{T:Donkin}}

For convenience we repeat the statement of this theorem below.

\setcounter{section}{1}
\setcounter{thm}{0}
\begin{thm}[Donkin \protect{\cite{SDonkin}}]
\DonkinText
\end{thm}
\setcounter{section}{4}
\setcounter{thm}{6}

%Note that $Y(\lambda|p\mu)$ is characterized as the unique summand
%of $M(\lambda|p\mu)$ not appearing in any $M(\alpha|\beta)$
%such that $(\alpha|\beta) \rhd (\lambda|p\mu)$.
We shall prove the theorem by showing that parts (i), (ii) and (iii) of Theorem \ref{T:Donkin} hold when $Y(\lambda|p\mu)$
is as defined in Definition~\ref{D:signedYoung}. In fact part (iii) holds by definition, so
we may concentrate on parts (i) and (ii).
%Note that if $\mu = \varnothing$ and $\lambda$ is $p$-restricted
%then (iii) asserts that $Y(\lambda | \varnothing)$ is projective; in all other cases $Y(\lambda|p\mu)$ has
%a non-trivial vertex. To deal with the projective summands we require the following corollary
%of Lemma 4.6 in \cite{james1978representation}. Given $\alpha \in \P(n)$, let $\kappa_\alpha = \sum_{g \in \sym{\alpha}}
%\sgn(g) g$.

\begin{proof}[Proof of Theorem \ref{T:Donkin}]
We work by induction on $n \in \N_0$. If $n < p$ then $F\sym{n}$ is semisimple
and the modules $Y(\lambda|\varnothing)$
for $\lambda \in \P(n)$ form a complete set of simple $F\sym{n}$-modules.
Hence parts (i) and~(ii) follow from Proposition~\ref{P:dominates}.
Now let $n \ge p$.

We first deal with non-projective summands.
Let $(\lambda | p\mu) \in \P^2(n)$ and
suppose that either $\lambda$ is not $p$-restricted or $\mu \not= \varnothing$.
Let $n_0 = |\lambda(0)|$ and let $n_i = |\lambda(i)| + |\mu(i-1)|$ for $i \in \N$.
Let $\rho = (1^{n_0},p^{n_1},\ldots,(p^r)^{n_r})$.

By Theorem~\ref{BC1} and Proposition~\ref{P:Broueperm},
$[M(\alpha|\beta) : Y(\lambda|p\mu)]$ is equal
to the sum of the products
%we have
%\[ = \sum_{(\mbf{\gamma}|\mbf{\delta})
%\in \Lambda(\alpha|\beta, \rho)}
\[ \bigl[W_1(\mbf{\gamma}_0|\mbf{\delta}_0) : P^{\lambda(0)}\bigr]
\bigl[W_p(\mbf{\gamma}_1|\mbf{\delta}_1) : Q_p\bigl(\lambda(1)|\mu(0)\bigr)\bigr] \cdots
\bigl[W_{p^r}(\mbf{\gamma}_r|\mbf{\delta}_r) : Q_{p^r}\bigl(\lambda(r)|\mu(r-1)\bigr)\bigr]
\]
over all $(\mbf{\gamma}|\mbf{\delta})\in \LambdaABR$.
Suppose the product is non-zero for $(\mbf{\gamma}|\mbf{\delta})\in \LambdaABR$.
Then $P^{\lambda(0)}$ is a direct summand of $W_1(\mbf{\gamma}_0|\mbf{\delta}_0)
\cong M(\mbf{\gamma}_0|\mbf{\delta}_0)$. Since $P^{\lambda(0)} = Y(\lambda(0)|\varnothing)$,
it follows from the inductive hypothesis that
$(\lambda(0)|\varnothing) \unrhd (\mbf{\gamma}_0|\mbf{\delta}_0)$.
Similarly $Q_{p^i}\bigl(\lambda(i)|\mu(i-1)\bigr)$ is a direct summand of $W_{p^i}(\mbf{\gamma}_i | \mbf{\delta}_i)$
for each $i \in \{1,\ldots,r\}$.
By Proposition~\ref{P:dominates}, we have $\lambda(i) \unrhd \mbf{\gamma}_i$
and $\mu(i-1) \unrhd \mbf{\delta}_i$
for each such $i$.
Hence
\begin{align}\label{eq:lambdaineq} \lambda - \lambda(0)  &= \sum_{i=1}^r p^i \lambda(i)
\unrhd \sum_{i=1}^r p^i \mbf{\gamma}_i = \alpha - \mbf{\gamma}_0
\intertext{and}
\label{eq:muineq} p\mu &= \sum_{i=1}^r p^i \mu(i-1) \unrhd \sum_{i=1}^r p^i \mbf{\delta}_i = \beta - \mbf{\delta}_0. \end{align}
Hence $\lambda \unrhd \alpha - \mbf{\gamma}_0 + \lambda(0) \unrhd \alpha$
and
\[ |\lambda| + \sum_{i=1}^j p\mu_i = |\alpha| + |\mbf{\delta}_0| + \sum_{i=1}^j p \mu_i
\ge |\alpha| + |\mbf{\delta}_0| + \sum_{i=1}^j (\beta - \mbf{\delta}_0)_i
\ge |\alpha| + \sum_{i=1}^j \beta_i \]
for all $j \in \N$.
Therefore $(\lambda|p\mu) \unrhd (\alpha|\beta)$.
By Proposition~\ref{P:summands}, every summand of  $M(\alpha|\beta)$ is isomorphic
to some $Y(\lambda|p\mu)$, so this
proves part (i) in the non-projective case.
If $(\alpha|\beta) = (\lambda | p\mu)$
then, by divisibility considerations,
$\mbf{\gamma}_0 = \lambda(0)$ and
$\mbf{\delta}_0 = \varnothing$.
Moreover, equality holds in both~\eqref{eq:lambdaineq}
and~\eqref{eq:muineq}, so we have $\mbf{\gamma}_i = \lambda(i)$ and $\mbf{\delta}_i = \mu(i-1)$
for each $i \in \{1,\ldots, r\}$.
Conversely, if $\mbf{\gamma}$ and $\mbf{\delta}$ are defined in this way, then
the product is $1$. This proves part (ii)
in the non-projective case.

We now deal with the projective summands.
By Proposition~\ref{P:summands}, if $\alpha \in \P(n)$ then~$M^\alpha$
is a direct sum of modules $Y(\lambda|\varnothing)$
for $\lambda \in \P(n)$.
The argument so far shows that
if $\alpha$ is not $p$-restricted then
$Y(\alpha|\varnothing)$ is a summand of $M^\alpha$, and $Y(\alpha|\varnothing)$
is a summand of $M^\gamma$
only if $\alpha \unrhd \gamma$. Therefore, inductively working down the dominance
order on partitions, we see that, for each such $\alpha$,
the submodule $U^\alpha$ in Proposition~\ref{P:Umodules}
is $Y(\alpha|\varnothing)$. By counting, the remaining
$U^\alpha$ for $\alpha \in \RP(n)$ are the modules
$Y(\lambda|\varnothing)$ for $\lambda \in \RP(n)$.
Again working inductively down the dominance order of partitions, it follows from
Proposition~\ref{P:dominates} that $U^\alpha = Y(\alpha|\varnothing)$
for each $\alpha \in \RP(n)$. This proves part (i) in the projective case
when $\beta =\varnothing$, and also proves part (ii) in the projective case.

Finally suppose that $\lambda$ is $p$-restricted and
$Y(\lambda|\varnothing)$ is a direct summand of $M(\alpha|\beta)$.
 Let $\t$ be a $\lambda$-tableau.
By Proposition~\ref{P:summands}, we have
$\kappa_\t M^{(\alpha|\beta)}  \not= 0$. Hence there exists an $(\alpha|\beta)$-tabloid
$\{\T\}$ such that $\kappa_\t\{\T\}\not=0$. By Lemma~\ref{lem:columnSym} we have
$(\lambda | \varnothing) \unrhd (\alpha|\beta)$. This completes the proof of part (i)
in the projective case.
\end{proof}

\section{Applications of Theorem~\ref{T:Donkin}}\label{S:Apps}

\subsection{Equivalent definitions}\label{S:EqvDefs}
We observed in the introduction that since signed Young modules
are characterized by Theorem~\ref{T:Donkin}, our definition
of signed Young modules agrees with Donkin's in \cite{SDonkin}.
Similarly Theorem~\ref{T:Donkin} characterizes the
Young module $Y(\lambda|\varnothing)$
as the unique summand of $M^\lambda$ appearing in $M^\mu$ only if $\lambda \unrhd \mu$.
By \cite[Theorem 3.1(i)]{JamesYoung}, James' Young modules admit the
same characterisation. The two definitions therefore agree.
In \cite{ErdSchr}, Erdmann and Schroll
consider Young modules for finite general linear groups. Adapting their proof
to symmetric groups (this is mentioned as a possibility in \cite{ErdSchr}, as a way to correct
\cite{ErdmannYoung}), their definition of the Young modules uses
the characterization in Proposition~\ref{P:summands}. Our proof of Theorem~\ref{T:Donkin}
shows these definitions agree; of course this also follows from the alternative
characterization just mentioned.

\begin{rem}{\ }
\begin{enumerate}
\item The counting argument used in our proof of the projective case of Theorem~\ref{T:Donkin} is
motivated by similar counting arguments used in \cite{ErdSchr}; the authors of \cite{ErdSchr} thank
Burkhard K{\"u}lshammer for suggesting this approach.

\item We have assumed throughout that $F$ has odd prime characteristic $p$.
It is possible
to construct Young modules when $p=2$ and to prove the analogue
of Theorem~\ref{T:Donkin} by adapting (and simplifying) the approach herein.

\item The analogue of signed Young modules for the finite general linear
group $\GL_n(\mathbb{F}_q)$ are the linear source modules induced from
powers of the determinant representation of parabolic
subgroups of~$\GL_n(\mathbb{F}_q)$. These modules seem worthy of study, especially
given the difficulty of working directly with Specht modules for $\GL_n(\mathbb{F}_q)$.
\end{enumerate}
\end{rem}

\subsection{Klyachko's formula and other applications}\label{S:Apps}
The following corollary generalizes Klyachko's formula to signed Young modules.
It is proved in the first step of our proof
of Theorem~\ref{T:Donkin}; alternatively it follows from this theorem
by taking Brou{\'e} correspondents.

\begin{cor}
\label{cor:Klyachko}
If $(\alpha|\beta)$ and $(\lambda|p\mu) \in \P^2(n)$ then
\[ \begin{split} &
[M(\alpha|\beta) : Y(\lambda|p\mu)] \\ &\qquad =
\!\! \sum_{(\mbf{\gamma}|\mbf{\delta})\in \Lambda((\alpha|\beta), \rho)}
\bigl[W_1(\mbf{\gamma_0}|\mbf{\delta}_0) : Y\bigl(\lambda(0) | \varnothing\bigr)\bigr]
\prod_{i=1}^r
\bigl[W_{p^i}(\mbf{\gamma_i}|\mbf{\delta}_i) : Q_{p^i}\bigl(\lambda(i)|\mu(i-1)\bigr)\bigr].
\end{split}
\]
\end{cor}

We remark that the reduction formula for signed $p$ Kostka numbers
 in Corollary \ref{cor:Klyachko} has previously been obtained by Danz, the first and the second authors in \cite{DGL}.

The proof of the following lemma is very easy and  is left to the reader.
Recall that the notation $\bullet$ for the concatenation of two compositions
was defined in~\ref{Sec:PCs}.

\begin{lem}\label{L:Nrhovarrho} Let
\begin{align*}
\rho&=(1^{m_0},p^{m_1},\ldots,(p^r)^{m_r}), \\ %\vdash m,\\
\gamma&=(1^{n_0},p^{n_1},\ldots,(p^s)^{n_s}), %\\ \vdash n,
\end{align*}
be partitions of $m$ and $n$ respectively, and let $k>r$. Then
\begin{align*}
  P_{\rho\bbullet p^k\gamma}&=P_\rho\times P_{p^k\gamma},\\
  N_{\sym{m+p^kn}}(P_{\rho}\times P_{p^k\gamma})&=N_{\sym{m}}(P_\rho)\times N_{\sym{p^kn}}(P_{p^k\gamma}),\\
  N_{\rho\bbullet p^k\gamma}/P_{\rho\bbullet p^k\gamma}&=(N_\rho/P_\rho)\times (N_{p^k\gamma}/P_{p^k\gamma}).
\end{align*}
\end{lem}

%Using Lemma~\ref{L:Nrhovarrho} we get the following corollary of Theorem~\ref{T:Donkin}.
Let $(\lambda|p\mu) \in \P^2(n)$. Suppose that the $p$-adic expansions of $\lambda$ and $\mu$
 are
 $\lambda = \sum_{i \ge 0} p^i \lambda(i)$ and
$\mu = \sum_{i \ge 0} p^i \mu(i)$, respectively. Let $\mu(-1) = \varnothing$.
If $r$ is maximal such that
$|\lambda(r)| + |\mu(r-1)| \not=0$ then we set
\begin{equation}\label{eq:ellp} \ell_p(\lambda|p\mu) = r. \end{equation}

\begin{lem}\label{L:twistBroue} Let $(\lambda|p\mu)\in\P^2(n)$ and let $P_\rho$ be a \vertex
of the signed Young module $Y(\lambda|p\mu)$.
\begin{enumerate}
\item [(i)] The signed Young module $Y(p\lambda|p^2\mu)$ has \vertex $P_{p\rho}$.
\item [(ii)] Suppose that $k>\ell_p(\lambda|p\mu)$ and let $(\alpha|\beta)\in\P^2(m)$ for some $m\in \N$. Then $Y(\lambda+p^k\pdot \alpha\,|\,p(\mu+p^k\pdot \beta))$ has \vertex $P_\rho\times P_{p^k
\pdot \gamma}$ where $P_\gamma$ is a \vertex of $Y(\alpha|p\beta)$. Moreover,
\[Y(\lambda|p\mu)(P_\rho)\boxtimes Y(p^k \pdot \alpha|p^{k+1}\pdot \beta)(P_{p^k\pdot \gamma})\]
is isomorphic to the Brou{\'e}
correspondent % of
$Y\bigl( \lambda+p^k \pdot \alpha\,|\,p(\mu+p^k \pdot \beta)\bigr)(P_\rho \times P_{p^k \pdot \gamma})$.
% with respect
%to $P_\rho \times P_{p^k \cdot \gamma}$.
\end{enumerate}
\end{lem}
\begin{proof} Suppose that $\lambda$ and $\mu$ have
$p$-adic expansions
$\sum_{i\geq 0} p^i\pdot \lambda(i)$ and $\sum_{i\geq 0} p^i\pdot \mu(i)$,
respectively. It is clear that the partitions $p\lambda$ and $p\mu$ have $p$-adic
expansions $p\lambda=\sum_{i\geq 1} p^i\pdot \lambda(i-1)$ and $p\mu=\sum_{i\geq 1} p^i\pdot \mu(i-1)$, respectively. So $|(p\lambda)(0)|=0$, and $|(p\lambda)(i)|+|(p\mu)(i-1)|=|\lambda(i-1)|+|\mu(i-2)|$ for all $i\geq 1$, where we set $\mu(-1)=\varnothing$. By Definition~\ref{D:signedYoung},
$Y(p\lambda|p^2\mu)$ has \vertex $P_{p\pdot \rho}$, proving part (i).
%\nu$ where \[\nu=(1^0,p^{|\lambda(0)|},(p^2)^{|\lambda(1)|+|\mu(0)|},\ldots)=p\rho.\]
%This shows that $P_\nu=P_{p\rho}$.

Let $r=\ell_p(\lambda|p\mu)$. For part (ii), since $k>r$, the $p$-adic expansions of $\lambda+p^k\alpha$ and $\mu+p^k\beta$ are
\begin{align*}
  \lambda+p^k\pdot \alpha&=\sum_{0\leq i\leq r} p^i\pdot \lambda(i)+\sum_{i\geq k} p^i\pdot \alpha(i-k),\\
  \mu+p^k\pdot \beta&=\sum_{0\leq i\leq r} p^i\pdot \mu(i)+\sum_{i\geq k} p^i\pdot \beta(i-k),
\end{align*} respectively. By Definition~\ref{D:signedYoung},
%Theorem \ref{T:Broue signed Young} shows that
$Y\bigl(\lambda+p^k\alpha\,|\,p(\mu+p^k\beta)\bigr)$ has \vertex $P_\eta$
where
\[\eta=\bigl(1^{|\lambda(0)|},p^{|\lambda(1)|+|\mu(0)|},\ldots,(p^r)^{|\lambda(r)|+|\mu(r-1)|},(p^k)^{|\alpha(0)|},(p^{k+1})^{|\alpha(1)|+|\beta(0)|},\ldots\bigr)=\rho\bbullet p^k\pdot \gamma.\]
Thus $P_\eta=P_{\rho\bbullet p^k\gamma}=P_\rho\times P_{p^k\gamma}$.
By Definition~\ref{D:signedYoung} and Lemma~\ref{L:Nrhovarrho}, we have
\begin{align*}
  Y(\lambda+p^k\alpha\,{}&{}|\,p(\mu+p^k\beta))(P_{\rho \bbullet p^k \pdot \gamma}) \\
  &=Q_1(\lambda(0)|\varnothing)\boxtimes Q_p(\lambda(1)|\mu(0))\boxtimes\cdots\boxtimes Q_{p^r}(\lambda(r)|\mu(r-1))\\
   &\quad\qquad\qquad \boxtimes Q_{p^k}(\alpha(0)|\varnothing)\boxtimes Q_{p^{k+1}}(\alpha(1)|\beta(0))\boxtimes\cdots\\
 % &=\mathbf{Q}(\lambda|p\mu)\boxtimes \mathbf{Q}(p^k\alpha|p^{k+1}\beta)\\
  &\cong Y\bigl(\lambda|p\mu)(P_\rho)\boxtimes Y(p^k\alpha|p^{k+1}\beta\bigr)(P_{p^k\gamma}),
\end{align*}
as required. %where the penultimate equation makes sense due to Lemma~\ref{L:Nrhovarrho}.
\end{proof}

The following result is an  interesting special case of \cite[Theorem 3.18]{SDanzKJLim}. It
is included  to illustrate a
technique used  again in the proof of Proposition~\ref{prop:labels}.

\begin{lem}\label{L:sgnTwist}
Let $n \in \N$. If $n = mp + c$ where $m \in \N_0$ and $0\leq c < p$
then $\sgn(n) \cong Y\bigl( (1^c) | (mp)\bigr)$.
\end{lem}

\begin{proof}
Let $n = \sum_{i=0}^r p^i n_i$ be the $p$-adic expansion of $n$,
and let $\rho = (1^{n_0},p^{n_1},\ldots,(p^r)^{n_r})$.
By Definition~\ref{D:signedYoung}, the signed Young module
$Y\bigl( (1^c) | (mp) \bigr)$ has $P_\rho$ as a vertex and
\[ Y\bigl( (1^c) | (mp) \bigr)(P_\rho) \cong Q_1\bigl( (1^c)|\varnothing\bigr) \boxtimes
Q_p\bigl( \varnothing | (n_1)\bigr) \boxtimes \cdots \boxtimes Q_{p^r}\bigl(\varnothing | (n_r)\bigr)\]
as a module for $F[N_\rho/P_\rho]$. Since $n_i < p$ we have
\[ Q_{p^i}\bigl( \varnothing | (n_i) \bigr) \cong \Inf_{\sym{n_i}}^{(N_{p^i}/P_{p^i}) \wr \sym{n_i}}
\bigl(F(n_i) \bigr)
\otimes \widehat{\sgn(N_{p^i})}^{\otimes n_i} \cong \Res^{\sym{p^in_i}}_{N_{p^i}\wr\sym{n_i}}
\bigl(\sgn(p^in_i)\bigr)   \]
where the second isomorphism follows from~\eqref{eq:sgnres}, regarding the right-hand
side as a representation of $(N_{p^i}/P_{p^i})\wr\sym{n_i}$. Hence there is an isomorphism of $FN_\rho$-modules,
\[ Y\bigl( (1^c) | (mp) \bigr)(P_\rho) \cong \Res^{\sym{n}}_{N_\rho}\bigl( \sgn(n) \bigr). \]
On the other hand, since $P_\rho$ is a Sylow $p$-subgroup of $\sym{n}$, it is a vertex
of $\sgn(n)$, and clearly $\sgn(n)(P_\rho) \cong \Res^{\sym{n}}_{N_\rho} \bigl(\sgn(n)\bigr)$
as an $FN_\rho$-module.
The Brou{\'e} correspondence is bijective (see Theorem~\ref{BC1}), so
we have $Y\bigl( (1^c) | (mp) \bigr) \cong \sgn(n)$.
\end{proof}

\section{Signed $p$-Kostka numbers}\label{S:Main}
In this section we prove
Theorem \ref{T:YmultMY} and Theorem \ref{T:2}. We work mainly with the
$F[(N_k/P_k) \wr \sym{m}]$-modules $W_k(\gamma|\delta)$ and $Q_k(\alpha|\beta)$ defined
in Definitions~\ref{D:VW} and~\ref{def:Qk}, and the $F[C_2 \wr \sym{m}]$-modules
$\overline{W}_k(\gamma|\delta)$ and $\overline{Q}_k(\alpha|\beta)$ obtained from them
by factoring out the trivial action of the even permutations in the base group of the wreath product.

% The main ingredients
% are the Brauer quotients of signed Young permutation modules
% and  the Brou\'{e} correspondents of indecomposable signed Young modules,
% as described in Sections~\ref{S:QuotientYoungPermutation} and~\ref{S:YoungConstruction}.
% We keep the notation used in these sections.
We begin with a key lemma for the proof of Theorem \ref{T:YmultMY}.

\begin{lem}\label{L:multWQ} Let $n \in \N$.
For any $(\gamma|\delta)\in\C^2(n)$ and $(\lambda|\mu)\in\RP^2(n)$
 we have
\begin{enumerate}
  \item [(i)] $\bigl[W_{p^{i+1}}(\gamma|\delta):Q_{p^{i+1}}(\lambda|\mu)\bigr]=
  \bigl[W_{p^i}(\gamma|\delta):Q_{p^i}(\lambda|\mu)\bigr]$ for all $i\geq 1$,
  \item [(ii)] $\bigl[W_p(\gamma|\varnothing):Q_p(\lambda|\varnothing)\bigr]=
  \bigl[W_1(\gamma|\varnothing):Q_1(\lambda|\varnothing)\bigr]$,
  \item [(iii)] $\bigl[W_p(\gamma|\delta):Q_p(\lambda|\varnothing)\bigr]=0$ if $\delta\neq \varnothing$.

  %$[W_p(\gamma|\delta):Q_p(\lambda|\mu)]\leq [W_1(\gamma|\delta):Q_1(\lambda|\mu)]$ with equality if and only if $\delta=\varnothing$.
\end{enumerate}
\end{lem}
\begin{proof} By Lemma \ref{L:barWk} and Lemma \ref{L:barQk}
we have $\overline{Q}_{p^j}(\lambda|\mu)\cong R_2(\lambda|\mu)$ and $\overline{W}_{\!p^j}(\gamma|\delta)\cong V_2(\gamma|\delta)$
 for all $j\geq 1$.
 Part (i) now follows by applying Lemma~\ref{L:trivialnormal}.
For (ii), if $\delta=\mu=\varnothing$ then
\begin{align*}
\overline{W}_{\!p}(\gamma|\varnothing)&\cong V_2(\gamma|\varnothing)= \Inf^{C_2\wr\sym{m}}_{\sym{m}}
(M^\gamma)=\overline{W}_{\!1}(\gamma|\varnothing),\\ \overline{Q}_p(\lambda|\varnothing)&\cong R_2(\lambda|\varnothing)=\Inf^{C_2\wr\sym{m}}_{\sym{m}}(\Plambda)=\overline{Q}_1(\lambda|\varnothing).
\end{align*}
So $\bigl[\vthinspace\overline{W}_{\!p}(\gamma|\varnothing):\overline{Q}_p(\lambda|\varnothing)\bigr]=
\bigl[\vthinspace\overline{W}_{\!1}(\gamma|\varnothing):\overline{Q}_1(\lambda|\varnothing)\bigr]$. Now apply
Lemma~\ref{L:trivialnormal}.
Finally (iii) follows from Proposition~\ref{P:dominates} and Lemma~\ref{L:trivialnormal}.
\end{proof}

We are now ready to prove Theorem \ref{T:YmultMY}. %For the reader's convenience we restate it below.

\begin{proof}[Proof of Theorem \ref{T:YmultMY}] Let $P_\rho$ be a \vertex of $Y(\lambda|p\mu)$.
By Definition~\ref{D:signedYoung} %repl , %Theorem~\ref{T:Broue signed Young}
we have
$$\rho=(1^{n_0},p^{n_1},(p^2)^{n_2},\ldots, (p^r)^{n_r}),$$ where $n_0=|\lambda(0)|$ and $n_i=|\lambda(i)|+|\mu(i-1)|$ for all $i\in\{1,\ldots,r\}$. By the Brou{\'e} correspondence (see Theorem~\ref{BC1})
and the description of the Brou{\'e} correspondents of signed Young modules in
Lemma~\ref{L:twistBroue},  it is equivalent to show
that
\[\bigl[M(p\alpha|p\beta)(P_{p\rho}):Y(p\lambda|p^2\mu)(P_{p\rho})\bigr]\leq
\bigl[M(\alpha|\beta)(P_\rho):Y(\lambda|p\mu)(P_\rho)\bigr].\]
Let $\Lambda = \LambdaABR$ and $\Lambda' =\Lambda\bigl((p\alpha|p\beta),p\rho\bigr)$
be as defined in Notation~\ref{N:Omega}. Observe that~$\Lambda'$
consists of all compositions
%\{ (\mbf{\gamma}'|\mbf{\delta}')=
\[ (\varnothing,\mbf{\gamma}_0,\mbf{\gamma}_1,\ldots,\mbf{\gamma}_r|\varnothing,\mbf{\delta}_0,\mbf{\delta}_1,\ldots,\mbf{\delta}_r) \]
where $(\mbf{\gamma}_0,\mbf{\gamma}_1,\ldots,\mbf{\gamma}_r|\mbf{\delta}_0,\mbf{\delta}_1,\ldots,\mbf{\delta}_r)\in\Lambda$. %\LambdaABR. \]
%Let $\Lambda = \LambdaABR$.
By Lemma~\ref{P:Broueperm} applied to $M(p\alpha|p\beta)(P_{p\rho})$, we have
\begin{align*}
M(p\alpha|p\beta)(P_{p\rho})&\cong \bigoplus_{(\mbf{\gamma}'|\mbf{\delta}')\in\Lambda'} W_1(\mbf{\gamma}'_0|\mbf{\delta}'_0)\boxtimes W_{\negthinspace p}(\mbf{\gamma}'_1|\mbf{\delta}'_1)\boxtimes \cdots\boxtimes W_{\!p^{r+1}}(\mbf{\gamma}'_{r+1}|\mbf{\delta}'_{r+1})\\
&= \bigoplus_{(\mbf{\gamma}|\mbf{\delta})\in\Lambda} W_1(\varnothing|\varnothing)\boxtimes W_{\negthinspace p}(\mbf{\gamma}_0|\mbf{\delta}_0)\boxtimes \cdots\boxtimes W_{\!p^{r+1}}(\mbf{\gamma}_r|\mbf{\delta}_r).
\end{align*}
By Definition~\ref{D:signedYoung} %repl
and Lemma~\ref{L:twistBroue}(i), we obtain both
\begin{align*}
\bigl[M(p\alpha|p\beta)(P_{p\rho}):Y(p\lambda|p^2\mu)(P_{p\rho})\bigr]
=&\sum_{(\mbf{\gamma}|\mbf{\delta})\in\Lambda}\prod_{i=0}^r
\bigl[W_{p^{i+1}}(\mbf{\gamma}_i|\mbf{\delta}_i):Q_{p^{i+1}}(\lambda(i)|\mu(i-1))\bigr],\\
\bigl[M(\alpha|\beta)(P_\rho):Y(\lambda|p\mu)(P_\rho)\bigr]=
&\sum_{(\mbf{\gamma}|\mbf{\delta})\in\Lambda}\prod_{i=0}^r \bigl[W_{p^i}(\mbf{\gamma}_i|\mbf{\delta}_i):Q_{p^i}(\lambda(i)|\mu(i-1))\bigr],
\end{align*} where, as usual, $\mu(-1)=\varnothing$. By Lemma~\ref{L:multWQ}, we have
\[\bigl[W_{p^{i+1}}(\mbf{\gamma}_i|\mbf{\delta}_i):Q_{p^{i+1}}(\lambda(i)|\mu(i-1))\bigr]=
\bigl[W_{p^i}(\mbf{\gamma}_i|\mbf{\delta}_i):Q_{p^i}(\lambda(i)|\mu(i-1))\bigr]\]
for all $i\geq 1$, and for $i=0$ whenever $\mbf{\delta}_0=\varnothing$. Otherwise, when $i=0$ and $\mbf{\delta}_0\neq\varnothing$, we have
\[0=\bigl[W_p(\mbf{\gamma}_0|\mbf{\delta}_0):Q_p(\lambda(0)|\varnothing)\bigr]\leq \bigl[W_1(\mbf{\gamma}_0|\mbf{\delta}_0):Q_1(\lambda(0)|\varnothing)\bigr].\]
This completes the proof.
\end{proof}

\begin{cor}\label{C:DELTAemptyset}
Let $(\alpha|\beta),(\lambda|p\mu)\in\P^2(n)$. Suppose that $\lambda(0)=\varnothing$. Then
\[\bigl[M(p\alpha|p\beta):Y(p\lambda|p^2\mu)\bigr]=\bigl[M(\alpha|\beta):Y(\lambda|p\mu)\bigr].\]
\end{cor}
\begin{proof}
Let $\rho \in \C(n)$ be defined by $$\rho=(1^{|\lambda(0)|}, p^{|\lambda(1)|+|\mu(0)|},\ldots, (p^r)^{|\lambda(r)|+|\mu(r-1)|}).$$
The \vertex $P_\rho$ of $Y(\lambda|p\mu)$ has no fixed points in
$\{1,2,\ldots, n\}$. Hence $\mbf{\delta}_0=\varnothing$ for any $(\mbf{\gamma}|\mbf{\delta})\in\LambdaABR$. The result now follows from
Theorem~\ref{T:YmultMY}.
\end{proof}

It is now very easy to deduce the asymptotic stability of signed $p$-Kostka numbers
mentioned in the introduction.

\begin{cor}\label{C:asymptotic}
Let $(\alpha|\beta),(\lambda|p\mu)\in\P^2(n)$. Then, for every natural number $k\geq 2$, we have
\[\bigl[M(p^k\alpha|p^k\beta):Y(p^k\lambda|p^{k+1}\mu)\bigr]=
\bigl[M(p\alpha|p\beta):Y(p\lambda|p^2\mu)\bigr]\leq \bigl[M(\alpha|\beta):Y(\lambda|p\mu)\bigr].\]
\end{cor}

\begin{proof}
This follows immediately from Corollary~\ref{C:DELTAemptyset} and Theorem~\ref{T:YmultMY}.
\end{proof}

\begin{eg}\label{Eg:strictineq} We present a family of examples
where the inequality in Theorem~\ref{T:YmultMY} is strict. Let $0<c<p$ and let $m \in \N$.
Since $\sym{mp}\times \sym{r}$ has index coprime to $p$ in $\sym{mp+c}$, the trivial module~$Y((mp+c)|\varnothing)$ is a direct summand of $M((mp,c)|\varnothing)$; the multiplicity is $1$ since
$M((mp,c)|\varnothing)$ comes from a transitive action of $\sym{mp+c}$.
By Lemma~\ref{L:sgnTwist} we have $\sgn(n) \cong Y((1^c)|(mp))$. Thus
\begin{align*}
\bigl[M\bigl(\varnothing|{}&{}(mp,c)\bigr):Y\bigl((1^c)|(mp)\bigr)\bigr] \\
&=\bigl[M\bigl(\varnothing|(mp,c)\bigr)\otimes\sgn(mp+c):Y((1^c)|(mp))\otimes \sgn(mp+c)\bigr]\\
&=\bigl[M\bigl((mp,c)|\varnothing\bigr):Y\bigl((mp+c)|\varnothing\bigr)\bigr]\\ &=1.
\intertext{On the other hand, $[M((mp^2,cp)|\varnothing):Y((mp^2)|(p(1^c)))]=0$ because, by \cite[2.3(6)]{SDonkin}, the signed Young modules are pairwise non-isomorphic and so
the signed Young module
$Y((mp^2)|p(1^c))$ is not isomorphic to a Young module. Thus we have}
\bigl[M\bigl(\varnothing|p({}&{}mp,c)\bigr):Y\bigl(p(1^c)|p(mp)\bigr)\bigr] \\
&=\bigl[M\bigl(\varnothing|(mp^2,cp)\bigr)\otimes\sgn(mp^2+cp):Y\bigl(p(1^c)|p(mp)\bigr)
\otimes \sgn(mp^2+cp)\bigr]\\
&=\bigl[M\bigl((mp^2,cp)|\varnothing\bigr):Y\bigl((mp^2)|p(1^c)\bigr)\bigr]\\ &=0,
\end{align*} where the penultimate equation is obtained using \cite[Theorem 3.18]{SDanzKJLim}. This shows that
\[ \bigl[M\bigl(\varnothing|p(mp,c)\bigr):Y\bigl(p(1^c)|p(mp)\bigr)\bigr]
<\bigl[M\bigl(\varnothing|(mp,c)\bigr):Y\bigl((1^c)|(mp)\bigr)\bigr].\]
\end{eg}

%\begin{ques} Is it true that whenever $\lambda(0)\neq\varnothing$ there exist $(\mbf{\gamma}|\mbf{\delta})\in\Omega$ with $\mbf{\delta}_0\neq\varnothing$ and hence further imply that the inequality in Theorem~\ref{T:YmultMY} is strict?
%\end{ques}
%\begin{proof}
%No. Take $M((1)|(3))(P_3)$ and $Y((4)|\varnothing)$. Here $\lambda(0)=(1)$ but $\delta(0)=\varnothing$.
%\end{proof}

We now turn to the proof of Theorem \ref{T:2}. We  need a
further result on the Brauer quotients of signed Young permutation modules.

\begin{prop}\label{P: 3.1} Let $m,n\in\N$ and let $(\pi|\widetilde\pi)\in \C^2(m)$ and $(\phi|\widetilde\phi)\in\C^2(n)$.
Let $\rho \in \C(m)$ and $\gamma \in \C(n)$ be compositions of the form
\begin{align*}
\rho&=(1^{m_0},p^{m_1},\ldots,(p^r)^{m_r}), \\ % \vdash m,\\
\gamma&=(1^{n_0},p^{n_1},\ldots,(p^s)^{n_s}). %\vdash n,
\end{align*}
For all $k\in\mathbb{N}$ such that $k>r$, we have that $M(\pi|\widetilde\pi)(P_\rho)\boxtimes M(p^k\phi|p^k\widetilde\phi)(P_{p^k\gamma})$ is isomorphic to a direct summand of $M(\pi+p^k\phi|\widetilde\pi+p^k\widetilde\phi)(P_{\rho\bbullet p^k\gamma})$. Furthermore, if $p^k>\mathrm{max}\{\pi_1,\widetilde{\pi}_1\}$, then $$M(\pi|\widetilde\pi)(P_\rho)\boxtimes M(p^k\phi|p^k\widetilde\phi)(P_{p^k\gamma}) \cong M(\pi+p^k\phi|\widetilde\pi+p^k\widetilde\phi)(P_{\rho\bbullet p^k\gamma}),$$
as $F[N_{\sym{m+p^kn}}(P_{\rho\bbullet p^k\gamma})/P_{\rho\bbullet p^k\gamma}]$-modules.
\end{prop}

Note that, in Proposition \ref{P: 3.1}, while $\sum_{i=0}^r m_i p^i = m$ and $\sum_{i=0}^r n_ip^i = n$, these need not be the base $p$ expressions for either $m$ or $n$.

\begin{proof}%[Proof of Proposition~\ref{P: 3.1}]
Since $k>r$, by Lemma~\ref{L:Nrhovarrho}, we have
$$P_{\rho\bbullet p^k\gamma}=P_\rho\times P_{p^k\gamma}.$$
To ease the notation, we denote by $M$, $M_1$ and $M_2$ the modules $M(\pi+p^k\phi|\widetilde\pi+p^k\widetilde\phi)$, $M(\pi|\widetilde\pi)$ and $M(p^k\phi|p^k\widetilde\phi)$, respectively. Furthermore, let $P=P_{\rho\bbullet p^k\gamma}$. By Corollary~\ref{C:basisBroueM}, $M(P)$ has as a basis
the subset $\mathcal{B}$  of $\Omega(\pi+p^k\phi|\widetilde\pi+p^k\widetilde\phi)$ consisting of
all $\{\mathrm{R}\}$ such that $\mathrm{R}$ is a row standard
$(\pi+p^k\phi|\widetilde\pi+p^k\widetilde\phi)$-tableau whose rows are unions of $P$-orbits. Similarly, we define  bases $\mathcal{B}_1$ and $\mathcal{B}_2$ of $\Omega(\pi|\widetilde\pi)$ and $\Omega(p^k\phi|p^k\widetilde\phi)$ for $M_1(P_\rho)$ and $M_2(P_{p^k\gamma})$, respectively; here each
 $(p^k\phi|p^k\widetilde\phi)$-tableau $\mathrm{S}$ of $\mathcal{B}_2$
 is filled with the numbers $m+1,m+2,\ldots,m+p^kn$.

For  $\{\T\}\in \mathcal{B}_1$ and
$\{\mathrm{S}\}\in \mathcal{B}_2$, let
$$\psi:\mathcal{B}_1\times\mathcal{B}_2\longrightarrow\mathcal{B}$$
be the map defined by $$\psi(\{\T\},\{\mathrm{S}\})=\bigl\{(\mathrm{R}_+| \mathrm{R}_-)\bigr\},$$
where $\mathrm{R}_+$ is the row standard $(\pi+p^k\phi)$-tableau such that row $i$ of $\mathrm{R}_+$ is the union of row~$i$ of $\T_+$ and row~$i$ of $\mathrm{S}_+$, and
 $\mathrm{R}_-$ is the row standard $(\widetilde\pi+p^k\widetilde\phi)$-tableau such that row~$i$ of $\mathrm{R}_-$ is the union of row $i$ of $\T_-$ and row $i$ of $\mathrm{S}_-$.
 Here we have used the convention row $i$ of $\T_+$ is empty if $i>\ell(\pi)$, and so on. The map $\psi$ is well defined since the rows of $\mathrm{R}=(\mathrm{R}_+|\mathrm{R}_-)$ are union of orbits of $P=P_\rho\times P_{p^k\gamma}$  on $\{1,2,\ldots, m+p^kn\}$.

Clearly $\psi$ is injective and so it induces an injection of vector spaces
\[\theta:M_1(P_\rho)\boxtimes M_2(P_{p^k\gamma})\longrightarrow M(P)\] defined by $\theta(\{\T\}\otimes\{\mathrm{S}\})=\psi(\{\T\},\{\mathrm{S}\})$. By Lemma~\ref{L:Nrhovarrho}, we may regard the domain and codomain of
$\theta$ as $FN_{\sym{m+p^kn}}(P)$-modules with trivial $P$-action.
It is not difficult to check that
 $$\theta(g(\{\T\}\otimes \{\mathrm{S}\}))=g\theta(\{\T\}\otimes \{\mathrm{S}\}),$$
for all $g\in N_{\sym{m+p^kn}}(P)$,  $\{\T\}\in \mathcal{B}_1$ and $\{\mathrm{S}\}\in\mathcal{B}_2$.
Therefore $\theta$ is an injective homomorphism
of $\F N_{\sym{m+p^kn}}(P)$-modules,
 and hence an injective homomorphism of  $F[N_{\sym{m+p^kn}}(P)/P]$-modules.
 Since both $M_1(P_\rho)$ and $M_2(P_{p^k\gamma})$ are projective and hence injective, their outer tensor product is also injective. Therefore, the map $\theta$ splits and we obtain that $M_1(P_\rho)\boxtimes M_2(P_{p^k\gamma})$ is a direct summand of $M(P)$.

The second assertion follows easily by observing that, if $p^k>\mathrm{max}\{\pi_1,\widetilde{\pi}_1\}$, then the map $\psi$ defined above is a bijection.
\end{proof}

We are now ready to prove Theorem \ref{T:2}.

\begin{proof}[Proof of Theorem \ref{T:2}]
Let $\rho \in \C(m)$ and $\gamma \in \C(n)$ be defined by %$$\rho=(1^{|\lambda(0)|}, p^{|\lambda(1)|+
\begin{align*}
\rho&=(1^{|\lambda(0)|}, p^{|\lambda(1)|+|\mu(0)|},\ldots, (p^r)^{|\lambda(r)|+|\mu(r-1)|}),\\
\gamma&=(1^{|\alpha(0)|}, p^{|\alpha(1)|+|\beta(0)|},\ldots, (p^s)^{|\alpha(s)|+|\beta(s-1)|}),
\end{align*} where $r=\ell_p(\lambda|p\mu)$ and $s=\ell_p(\alpha|p\beta)$, respectively.
By Definition~\ref{D:signedYoung}, %repl
$P_\rho$ is a \vertex of $Y(\lambda|p\mu)$ and $P_\gamma$ is a \vertex of $Y(\alpha|p\beta)$.
Since $k>r$, by Lemma \ref{L:Nrhovarrho}, we have
$P_{\rho\bbullet p^k\gamma}=P_\rho \times P_{p^k\gamma}$ and
$$ N_{\sym{m+p^kn}}(P_{\rho}\times P_{p^k\gamma})=N_{\sym{m}}(P_\rho)\times N_{\sym{p^kn}}(P_{p^k\gamma}).$$
By Lemma \ref{L:twistBroue},
$Y(p^k\alpha|p^{k+1}\beta)$ has \vertex $P_{p^k\gamma}$
and $Y(\lambda+p^k\alpha|p(\mu+p^k\beta))$ has \vertex $P_{\rho\bbullet p^k\gamma}$.
Moreover, the
Brou\'{e} correspondent of $Y(\lambda+p^k\alpha\,|\,p(\mu+p^k\beta))$ is $$Y(\lambda|p\mu)(P_\rho)\boxtimes Y(p^k\alpha|p^{k+1}\beta)(P_{p^k\gamma}).$$
By Proposition \ref{P: 3.1}, we have
$$M(\pi|\widetilde\pi)(P_\rho)\boxtimes M(p^k\phi|p^k\widetilde\phi)(P_{p^k\gamma})\
\big{|}\ M(\pi+p^k\phi\,|\,\widetilde\pi+p^k\widetilde\phi)(P_{\rho\bbullet p^k\gamma}).$$
Therefore, using Theorem \ref{BC1}(ii), we deduce that
\begin{align*}
\bigl[M(\pi{}&{}+p^k\phi\,|\,\widetilde\pi+p^k\widetilde\phi):Y(\lambda+p^k\alpha\,|\,p(\mu+p^k\beta))\bigr]\\
&= \bigl[M(\pi+p^k\phi\,|\,\widetilde\pi+p^k\widetilde\phi)(P_{\rho\bbullet p^k\gamma}):Y(\lambda+p^k\alpha\,|\,p(\mu+p^k\beta))(P_{\rho\bbullet p^k\gamma})\bigr]\\
&\geq  \bigl[M(\pi|\widetilde\pi)(P_\rho)\boxtimes M(p^k\phi|p^k\widetilde\phi)(P_{p^k\gamma}):Y(\lambda|p\mu)(P_\rho)\boxtimes Y(p^k\alpha|p^{k+1}\beta)(P_{p^k\gamma})\bigr]\\
&= \bigl[M(\pi|\widetilde\pi)(P_\rho):Y(\lambda|p\mu)(P_\rho)\bigr] \bigl[M(p^k\phi|p^k\widetilde\phi)(P_{p^k\gamma}):Y(p^k\alpha|p^{k+1}\beta)(P_{p^k\gamma})\bigr]\\
&= \bigl[M(\pi|\widetilde\pi):Y(\lambda|p\mu)\bigr] \bigl[M(p^k\phi|p^k\widetilde\phi):Y(p^k\alpha|p^{k+1}\beta)\bigr]\\
&= \bigl[M(\pi|\widetilde\pi):Y(\lambda|p\mu)\bigr] \bigl[M(p\phi|p\widetilde\phi):Y(p\alpha|p^2\beta)],
\end{align*}
where the final equality follows from Corollary \ref{C:asymptotic}.
If $p^k>\mathrm{max}\{\pi_1,\widetilde{\pi}_1\}$, then Proposition \ref{P: 3.1} implies that we have equalities throughout.
\end{proof}

\section{Indecomposable signed Young permutation modules}\label{S:Last}

In this section, in the spirit of Gill's result \cite[Theorem~2]{CGill}, we classify all indecomposable signed Young permutation modules over the field $F$
and determine their endomorphism algebras
and their labels as signed Young modules.
By \cite{CGill}, any
indecomposable
Young permutation module is of the form $M^{(m)}$ or $M^{(kp-1,1)}$. It is immediate
from the definition of signed Young permutation modules in~\eqref{eq:signedYoung} that
%In order to extend this classification to signed Young permutation modules and hence to prove Theorem \ref{T:IndecSigned}, a key observation is that
\[M(\alpha|\beta)\cong \Ind^{\sym{|\alpha|+|\beta|}}_{\sym{|\alpha|}\times\sym{|\beta|}}\big(M^\alpha\boxtimes (M^\beta\otimes\sgn(|\beta|)\big).\]
As such, by Gill's result, any indecomposable signed Young permutation module
is of one of the forms $M((m)|(n))$, $M((m)|(kp-1,1))$, $M((kp-1,1)|(m))$ or $M((kp-1,1)|(\ell p-1,1))$. Since $M((m)|(kp-1,1))\otimes \sgn(m+kp)\cong M((kp-1,1)|(m))$, there are essentially three
different forms to consider.

%\setcounter{section}{1}
%\setcounter{thm}{2}
%\begin{thm} Let $(\alpha|\beta)\in\P^2(n)$. A signed Young permutation module $M(\alpha|\beta)$ is indecomposable if and only if one of the following conditions holds.
%\begin{enumerate}
  %\item [(i)] $(\alpha|\beta)=((m)|(n))$ for some non-negative integers $m,n$ such that either
  %\begin{enumerate}
  %\item [(a)] $m=0$,
  %\item [(b)] $n=0$, or
  %\item [(c)] $m+n$ is divisible by $p$.
  %\end{enumerate}
  %\item [(ii)] $(\alpha|\beta)$ is either $((kp-1,1)|\varnothing)$ or $(\varnothing|(kp-1,1))$ for some positive integer $k$.
%\end{enumerate}

%The only indecomposable signed Young permutation modules are either $M((m)|(n))$, $M((kp-1,1)|\varnothing)$ or $M(\varnothing|(kp-1,1))$, where $m,n$ are non-negative integers and $k$ is a positive integer.
%\end{thm}
\begin{proof}[Proof of Theorem~\ref{T:IndecSigned}]
%We consider the three forms $M((m)|(n))$, $M((kp-1,1)|(m))$ and $M((kp-1,1)|(\ell p-1,1))$.
Let $M_1=M((m)|(n))$. If $m = 0$ then $M_1$ is the sign representation, and if $n = 0$ then $M_1$ is the trivial representation. In these cases,~$M_1$ is simple with %MW: indecomposable -> simple, endo
$1$-dimensional endomorphism ring. Suppose that both $m,n$ are non-zero. By the Littlewood--Richardson rule,
the module $M_1$ has a Specht series with top Specht factor $S^{(m+1,1^{n-1})}$ and
bottom Specht factor $S^{(m,1^n)}$. If $m+n$ is not divisible by $p$ then the $p$-cores of $(m+1,1^{n-1})$ and $(m,1^n)$ are non-empty and distinct
and so $S^{(m+1,1^{n-1})}$ and $S^{(m,1^n)}$ lie in different blocks. Consequently, $M_1$ is decomposable. Now suppose that $m+n$ is divisible by $p$. In this case, by Peel's result \cite{PeelHooks},
\begin{align*}
S^{(m+1,1^{n-1})}&=\left \{\begin{array}{ll} F&n=1,\\ \begin{bmatrix} D^\lambda\\ D^\gamma\end{bmatrix}&n\geq 2, \end{array}\right .& S^{(m,1^n)}&=\left \{\begin{array}{ll} \sgn(m+n)&m=1,\\ \begin{bmatrix} D^\mu\\ D^\lambda\end{bmatrix}&m\geq 2, \end{array}\right .
\end{align*} where $\mu,\lambda,\gamma$ are the $p$-regularization of the partitions $(m,1^n)$, $(m+1,1^{n-1})$ and $(m+2,1^{n-2})$ respectively (see \cite[6.3.48]{JK}). If $m=1$ then $M((1)|(n))\cong M(\varnothing|(n,1))$ is indecomposable. Similarly, if $n=1$ then $M((m)|(1))\cong M((m,1)|\varnothing)$ is indecomposable.
Moreover, since $M((m,1)|\varnothing)$ has a Loewy series with factors $F, D^{(m,1)}, F$,
the endomorphism algebra
$\End_{F\sym{m+1}}M((m,1)|\varnothing)$ is $2$-dimensional. Tensoring by the sign representation
we obtain the same result for $\End_{F\sym{n+1}}(\varnothing|(n,1))$. %MW endos

We now study the case
when $m,n \ge 2$. In this case, both the head and socle of~$M_1$ contain the simple module $D^\lambda$. Also, as a signed Young permutation module, $M_1$ is self-dual. Suppose that $D^\gamma$ is not isomorphic to a composition factor of any direct summand of $M_1$ containing $D^\lambda$ in its head (and hence in its socle). Then $D^\gamma$ is necessarily isomorphic to a direct summand of $M_1$.
From the Specht series, there is a surjection $\psi$ from $M_1$ onto the Specht module $S=S^{(m+1,1^{n-1})}$. Since $S$  has composition factors $D^\gamma$ and $D^\lambda$, we have $\psi(D^\gamma)\neq 0$ and so $\psi(D^\gamma)\cong D^\gamma$. Let $Y$ be an indecomposable direct summand of $M_1$ such that $\psi(Y)$ contains a composition factor $D^\lambda$. This shows that $\psi(Y)\cong D^\lambda$ and hence \[S=\psi(D^\gamma\oplus Y)\cong D^\gamma\oplus Y/(Y\cap \ker\psi)\cong D^\gamma\oplus D^\lambda.\] This is absurd since $S$ is indecomposable.
Hence there exists an indecomposable direct summand of $M_1$ containing $D^\lambda$ in its head and that does not contain $D^\gamma$ in its head or in its socle. Dually, there exists an indecomposable direct summand of $M_1$ containing $D^\lambda$ in its head, that does not contain $D^\mu$ in its head or in its socle. Thus the only possibility is that $M_1$ is indecomposable with the Loewy structure \[\begin{bmatrix} D^\lambda\\ D^\mu\ \ D^\gamma\\ D^\lambda\end{bmatrix}\]
and has $2$-dimensional endomorphism ring. %MW endos

Let $M_2=M((kp-1,1)|(m))$. By Gill's result, if
 $m=0$ then $M_2$ is indecomposable and
if $m=1$ then $M_2\cong M((kp-1,1^2)|\varnothing)$ is decomposable.
Suppose that $m\geq 2$. By the Littlewood--Richardson rule, $M_2$ has a Specht series with
Specht factors
$S_1=S^{(kp+1,1^{m-1})}$, $S_2=S^{(kp,2,1^{m-2})}$, $S_3=S^{(kp,1^m)}$,
$S_4 = S^{(kp-1,2,1^{m-1})}$, $S_5=S^{(kp-1,1^{m+1})}$,
%\begin{align*}
%  S_1&=S^{(kp+1,1^{m-1})},&S_2&=S^{(kp,2,1^{m-2})},&S_3&=S^{(kp,1^m)},\\ S_4&=S^{(kp-1,2,1^{m-1})},&S_5&=S^{(kp-1,1^{m+1})},
%\end{align*}
 with $S_3$ occurring twice. If $m\not\equiv 0$ mod $p$, then $S_1$ and $S_3$
lie in different blocks. If $m\equiv0$ mod $p$ then $S_3$ and $S_4$ belong to different blocks. Thus we conclude that $M_2$ is decomposable whenever $m\geq 2$.

Let $M_3=M((kp-1,1)|(\ell p-1,1))$. Then $M_3\cong M((kp-1,1^2)|(\ell p-1))$. By Gill's result, since $M^{(kp-1,1^2)}$ is decomposable, we have that \[M\bigl( (kp-1,1^2)|(\ell p-1) \bigr)
=\Ind^{\sym{kp+\ell p}}_{\sym{kp+1}\times\sym{\ell p-1}}\bigl(M^{(kp-1,1^2)}\boxtimes
(M^{(\ell p-1)}\otimes \sgn(\ell p-1))\bigr)\] is decomposable.
\end{proof}

We end by determining the labels of the indecomposable signed Young permutation modules.
By the remark immediately following the statement of Theorem~\ref{T:IndecSigned},
it suffices to consider the modules $M((m) | (n))$ where either $m=0$, $n=0$ or $m+n$ is divisible
by $p$.

\begin{prop}\label{prop:labels}
Let $m$, $n \in \N$. Let $n = n_0 + pn'$ where $0\leq n_0 < p$.
There are isomorphisms $M((m)|\varnothing) \cong Y((m)|\varnothing)$,
$M(\varnothing|(n)) \cong Y((1^{n_0})|(pn'))$ and, provided $m+n$ is divisible by $p$,
$M((m)|(n)) \cong Y((m,1^{n_0})|(pn'))$.
\end{prop}

\begin{proof}
Clearly $M((n) |\varnothing) \cong Y((n) | \varnothing) \cong F(n)$.
The second isomorphism follows from Lemma~\ref{L:sgnTwist}.
In the remaining case, $m$, $n > 0$ and  $m+n$ is divisible by $p$.
Let $m = \sum_{i \ge 0} m_i p^i$ and let $n = \sum_{i \ge 0} n_i p^i$ be the $p$-adic expansions.
Let $r$ be the greatest such that $m_r+n_r \not=0$.
Let $P$ be a Sylow $p$-subgroup of $\sym{m} \times \sym{n}$.
By Proposition~\ref{P:Broueperm} we have an isomorphism of $F[N_{\sym{m+n}}(P)/P]$-modules
\[ M\bigl((m) | (n)\bigr)(P) \cong W_1\bigl((m_0) | (n_0)\bigr) \boxtimes W_p\bigl((m_1)|(n_1)
\bigr) \boxtimes
\cdots \boxtimes W_{p^r}\bigl((m_r) | (n_r)\bigr). \]
By Theorem~3.3, the signed Young module
$Y((m,1^{n_0})|(pn'))$ satisfies
\[ Y\bigl((m,1^{n_0})|(pn')\bigr)(P) =
Y((m_0,1^{n_0}) | \varnothing) \boxtimes Q_p((m_1) | (n_1)) \boxtimes
\cdots \boxtimes Q_{p^r}((m_r) | (n_r)), \]
where $Q_{p^i}((m_i)|(n_i))$
is the $F[(N_k/P_k) \wr \sym{m}]$-module
defined in Definition~\ref{def:Qk}.
The Brou{\'e} correspondence is bijective (see Theorem~\ref{BC1}), so
it suffices to prove that the tensor factors in these two modules agree.

Observe that $m_0 + n_0$ is a multiple of $p$ and $m_0 + n_0 < 2p$.
If $m_0 = n_0 = 0$ we have $W_1(\varnothing|\varnothing)=Y(\varnothing|\varnothing)$. Next, we assume that
%there is nothing left to prove, so we may assume that
$m_0 + n_0 = p$. The $F\sym{p}$-module
$W_1((m_0) | (n_0)) \cong M((m_0) | (n_0))$ is
indecomposable by Theorem~\ref{T:IndecSigned}.
The only signed Young module for $F\sym{p}$
that is not a Young module is the sign representation. Since $n_0 < p$ we see
that $M((m_0) | (n_0))$ is a Young module. The proof of Theorem~\ref{T:IndecSigned}
shows that it has a Specht filtration with $S^{(m_0,1^{n_0})}$ at the bottom and
$S^{(m_0+1,1^{n_0-1})}$ at the top.
Therefore $W_1((m_0)|(n_0))=M((m_0)|(n_0)) \cong Y((m_0,1^{n_0})|\varnothing)$, as required.

Finally suppose that $i \ge 1$.
By Definition~\ref{D:VW}(ii) $W_{p^i}((m_i)|(n_i))$ is
the $\F[(N_{p^i}/P_{p^i}) \wr \sym{m_i+n_i}]$-module
obtained from
\[ %W_{p^i}((m_i)|(n_i)) \cong
\Ind_{N_{p^i}\wr(\sym{m_i}\times \sym{n_i})}^{N_{p^i}\wr\sym{m_i+n_i}}\left
(\Inf^{N_{p^i}\wr\sym{m_i}}_{\sym{m_i}}(F(m_i))\boxtimes \bigl(
(\Inf^{N_{p^i}\wr\sym{n_i}}_{\sym{n_i}}(F(n_i))\otimes \widehat{\sgn(N_{p^i})}^{\otimes n_i}\bigr)\right).
\] %NB this is a module for N/P
by the canonical surjection $(N_{p^i} \wr \sym{m_i+n_i}) / (P_{p^i})^{m_i+n_i}
\cong (N_{p^i} / P_{p^i}) \wr \sym{m_i+n_i}$.
Since $m_i, n_i < p$ the projective covers $P^{(m_i)}$ and $P^{(n_i)}$ are the
trivial $F\sym{m_i}$- and $F\sym{n_i}$-modules, respectively.
Therefore, by~\eqref{eq:Qk}, we have
$W_{p^i}((m_i)|(n_i)) \cong Q_{p^i}((m_i)|(n_i))$, again as required.
\end{proof}

\end{document}